\documentclass{mathformatv1}

\usepackage{tikz-cd}
\usetikzlibrary{arrows}
\usetikzlibrary{decorations.markings}
\usetikzlibrary{matrix}
\usetikzlibrary{patterns}
\usetikzlibrary{positioning}
\usepackage{array}

\def\g{\mathfrak{g}}
\def\gl{\mathfrak{gl}}

\def\tr{\mathrm{Tr}}
\def\C{\mathcal{C}}

\def\h{\mathfrak{h}}

\def\H{\mathcal{H}}

\def\End{\mathrm{End}}
\def\Hom{\mathop{\mathrm{Hom}}\nolimits}
\def\ker{\mathop{\mathrm{ker}}\nolimits}

\def\gr{\mathop{gr}}
\def\Lie{\mathop{\mathrm{Lie}}\nolimits}

\def\Z{\mathbb{Z}}

\def\n{\mathfrak{n}}
\def\m{\mathfrak{m}}
\def\ad{\mathrm{ad}}
\def\Ext{\mathop{\mathrm{Ext}}\nolimits}

\def\O{\mathcal{O}}

\def\Spec{\mathop{\mathrm{Spec}}\nolimits}

\def\Ver{\mathrm{Ver}}
\def\Vec{\mathrm{Vec}}
\def\ind{\mathrm{ind}}
\def\pro{\mathrm{pro}}
\def\Ann{\mathop{\mathrm{Ann}}\nolimits}
\def\m{\mathfrak{m}}

\def\H{\mathcal{H}}
\def\A{\mathcal{A}}
\def\I{\mathcal{I}}
\def\sVec{\mathrm{sVec}}
\def\k{\mathbf{k}}
\def\id{\mathrm{id}}

\def\Ker{\mathop{\mathrm{Ker}}\nolimits}

\def\Ann{\mathop{\mathrm{Ann}}\nolimits}

\def\Ver{\mathrm{Ver}}
\def\ev{\mathrm{ev}}
\def\coev{\mathrm{coev}}
\def\ind{\mathrm{ind}}
\def\pro{\mathrm{pro}}
\def\Bl{\mathop{\mathrm{Bl}}\nolimits}

\def\alg{\mathrm{alg}}

\def\Prim{\mathop{\mathrm{Prim}}\nolimits}
\def\FOLie{\mathop{\mathrm{FOLie}}\nolimits}
\def\H{\mathcal{H}}
\def\iHom{\mathop{\underline{\mathrm{Hom}}}\nolimits}
\def\Corad{\mathop{\mathrm{Corad}}\nolimits}

\def\susbeteq{\subseteq}

\begin{document}

\title{Harish-Chandra pairs in the Verlinde category in positive characteristic}
\author{Siddharth Venkatesh}

\maketitle

\begin{abstract}

In this article, we develop the theory of commutative and cocommutative Hopf algerbas in the Verlinde category and prove that the category of affine group schemes of finite type in the Verlinde category is equivalent to the category of Harish-Chandra pairs in the Verlinde category. Subsequently, we extend this equivalence to an equivalence between corresponding representation categories 

\end{abstract}

\maketitle

\section{\Large{\textbf{Introduction}}}

Fix an algebraically closed field $\mathbf{k}$ of characteristic $p > 0$. The Verlinde category $\Ver_{p}$ is the semisimplification of the category of finite dimensional $\k$-representations of $\Z/p\Z$. It can also be constructed as the semisimplification of the category of tilting $SL_{2}$-modules over $\k$. This is a symmetric fusion category over $\k$ that is a universal base for all such categories. More precisely, we have the following theorem of Ostrik (\cite{O}):

\begin{theorem} Let $\C$ be any symmetric fusion category over $\k$. Then there exists a symmetric tensor functor $F: \C \rightarrow \Ver_{p}$.

\end{theorem}

A consequence of this theorem is that if $\C$ is any $\k$-linear symmetric tensor category (\cite{EGNO} for the definition) fibered over a symmetric fusion category, then it is equivalent to the category of representations of some affine group scheme in $\Ver_{p}$, i.e., it is equivalent to the category of comodules of some commutative ind-Hopf algebra in $\Ver_{p}$. The goal of this paper is to better understand such Hopf algebras and their comodule categories. 

To do so, we will relate them to objects in $\Ver_{p}$ that are slightly easier to work with algebraically and combinatorially. These will be what we call Harish-Chandra pairs in $\Ver_{p}$. Roughly speaking, a Harish-Chandra pair in $\Ver_{p}$ is the data of an affine group scheme $G_{0}$ of finite type over $\k$, a Lie algebra $\g$ in $\Ver_{p}$ such that $\g_{0} = \Lie(G_{0})$, along with an extension of the adjoint action of $G_{0}$ on $\g_{0}$ to an action of $G_{0}$ on $\g$. We will give a more precise formal definition of a Harish-Chandra pair in the relevant section in the paper but this one here suffices for us to be able to state the main results of the paper. 

\begin{theorem} \label{Harish-Chandra} The category of affine group schemes of finite type in $\Ver_{p}$ is equivalent to the category of Harish-Chandra pairs in $\Ver_{p}$. This equivalence sends an affine group scheme $G$ of finite type in $\Ver_{p}$ to $(G_{0}, \Lie(G))$, where $G_{0}$ is the underlying ordinary affine group scheme associated to $G$ and $\Lie(G)$ is the Lie algebra of $G$.

\end{theorem}

Let us call the functor assigning a Harish-Chandra pair to an affine group scheme in $\Ver_{p}$ the Harish-Chandra functor and denote it by $\mathbf{HC}$. This theorem essentially states that all the new $\Ver_{p}$ specific behavior of an affine group scheme in $\Ver_{p}$ comes from its Lie algebra. Since the Lie algebra of an affine group scheme of finite type in $\Ver_{p}$ is an object of  finite length, in contrast to the possibly infinite length commutative Hopf algebra of functions, it tends to be significantly easier to work with algebraically. Hence, this theorem greatly simplifies the study of affine group schemes of finite type in $\Ver_{p}$. One particular consequence of the theorem is a correspondence between closed subgroups and Lie subalgebras. 

\begin{corollary} \label{subgroups} Let $G$ be an affine group scheme of finite type in $\Ver_{p}$ and let $(G_{0}, \g)$ be the corresponding Harish-Chandra pair in $\Ver_{p}$. The Harish-Chandra functor establishes a bijection between the set of closed subgroups of $G$ and the set 
$$\{(G'_{0}, \g') : G'_{0} \text{ a closed subgroup of } G_{0}, \; \g' \text{ a Lie subalgebra of } \g \text{ with } \Lie(G'_{0}) = \g'_{0}\}.$$

\end{corollary}

This theorem also extends to an equivalence between representation categories. A representation of $G$ in $\Ver_{p}$ is simply a comodule for $\O(G)$, the commutative ind-Hopf algebra defining $G$. A representation for a Harish-Chandra pair $(G_{0}, \g)$ in $\Ver_{p}$ is an object in $\Ver_{p}$ equipped simultaneously with an action of $G_{0}$ and $\g$ such that the $\g$-action map is $G_{0}$-linear and such that the two restrictions to $\g_{0}$ are the same.

\begin{corollary} \label{repHC} Let $G$ be an affine group scheme of finite type in $\Ver_{p}$ and let $(G_{0}, \g)$ be the corresponding Harish-Chandra pair in $\Ver_{p}$. Then, the category of representations of $G$ in $\Ver_{p}$ is equivalent to the category of representations of $(G_{0}, \g)$ in $\Ver_{p}$.

\end{corollary}

This corollary is very important because it allows us to construct representations for affine group schemes $G$ in $\Ver_{p}$ via representations of ordinary algebraic groups and representations for the Lie algebra, both of which are much easier to construct. Additionally, this corollary says that if we have a representation for the Lie algebra of $G$, then testing whether that representation integrates to the group is the same as seeing if the representation integrates to $G_{0}$. This will be very useful in a follow up paper to classify irreducible representations of certain simple affine group schemes in $\Ver_{p}$.

This paper is organized as follows. In section 2, we give a detailed construction of the Verlinde category and restate some key results from \cite{Ven1} regarding finitely generated commutative algebras in $\Ver_{p}$. In section 3, we build these results to prove some more fundamental commutative algebra results in $\Ver_{p}$ that will prove useful in relating commutative Hopf algebras with their dual coalgebras. In section 4, we develop some of the theory of cocommutative ind-coalgebras in $\Ver_{p}$, with a particular focus on coradical filtrations and relative coradical filtrations, as well as structure of the dual coalgebra of a commutative ind-algebra. In section 5, we define dual Harish-Chandra pairs and Harish-Chandra pairs in $\Ver_{p}$ and the functors from cocommutative ind-Hopf algebras to dual Harish-Chandra pairs and from commutative ind-Hopf algebras to Harish-Chandra pairs in $\Ver_{p}$. In section 6, we study some PBW properties of cocommutative ind-Hopf algebras and dual Harish-Chandra pairs in $\Ver_{p}$ and use this to prove the equivalence between these two categories. In section 7, we establish some dualities between the cocommutative and the commutative setting and use this to prove Theorem \ref{Harish-Chandra} and its corollaries. Finally, in section 8, we use this theorem to prove Corollary \ref{repHC}.

\subsection{Acknowledgements:}  I am very grateful to my advisor Pavel Etingof for both suggesting the problems studied in this paper and providing a large amount of helpful advice in how to approach the proofs of the main theorems. The work in this paper also owes a large debt to the work of Akira Masuoka on Harish-Chandra pairs in the category of supervector spaces over $\k$ (\cite{M1}, \cite{M2}). Additionally, I am grateful to the referees for numerous insightful comments. The proof of lemma \ref{smoothness} in particular was made much clearer thanks to input from the referees. 

\section{\Large{\textbf{Technical Background}}}

\subsection{Notation and Conventions}

These notations and conventions will be brought up in the relevant sections as well but are all stated here for convenience of reader.

\begin{enumerate}

\item[1.] Unless specified otherwise, $\k$ will be an algebraically closed field of characteristic $p>0$.

\item[2.] By a category over $\k$, we mean a $\k$-linear, locally finite and Artinian category.

\item[3.] If $\C$ is a symmetric tensor category, we will always use $c$ to denote the braiding on $\C$. When the objects on which the braiding is acting need to be specified, we will explicitly write $c_{X, Y}$ instead of $c$.

\item[4.] In comparison between a symmetric tensor category $\C$ and its ind-completion $\C^{\ind}$, we will use the word ``object'' to mean an object in $\C$, i.e., one of finite length, and we will use the phrase ``ind-object'' to refer more generally to an object in $\C^{\ind}$, one that may possibly be of infinite length. 

\item[5.] For objects $X$ inside $\Ver_{p}^{\ind}$, we will use $X_{0}$ to denote the isotypic component corresponding to the monoidal unit $\mathbf{1}$, and $X_{\not=0}$ to denote the sum of all other isotypic components.

\item[6.] We will consistently use $C$ to denote cocommutative coalgebras and Hopf algebras in $\Ver_{p}^{\ind}$ and $A$ to denote commutative algebras and Hopf algebras in $\Ver_{p}^{\ind}$. We will use $J$ to denote cocommutative coalgebras and Hopf algebras over $\k$ and $H$ to denote commutative algebras and Hopf algebras over $\k$.

\end{enumerate}

\subsection{Tensor Category Technicalities}

For definitions of tensor categories, braided tensor categories, symmetric tensor categories and symmetric tensor functors the reader is referred to \cite{EGNO}. Here, we present some examples.

\begin{example}

\begin{enumerate}

\item[(a)] The simplest examples of symmetric tensor categories are $\Vec$ and $\sVec$ which are, respectively, the categories of finite dimensional $\k$-vector spaces and finite dimensional $\k$-vector superspaces. Here, the braiding is just the swap map and the signed swap map, respectively.

\item[(b)] Similarly, the category of finite dimensional representations over $\k$ of a finite group $G$ is a symmetric finite tensor category over $\k$ with braiding given by the swap map. More generally, the category of finite dimensional comodules over a commutative Hopf algebra $H$ is a symmetric tensor category as well. 

\item[(c)] A slightly more complicated category is the universal Verlinde category in characteristic $p > 0$, which we denote as $\Ver_{p}$. This is constructed as a quotient of the category of finite dimensional representations of $\mathbb{Z}/p\mathbb{Z}$ over $\k$ of characteristic $p$. The full details regarding the construction are given in a later subsection.

\end{enumerate}

\end{example}

The main technical construction in this section that we need encapsulates the notion of infinite-dimensionality. We will be working with algebras inside a symmetric tensor category that are not necessarily ``finite dimensional'', since finite dimensional algebras tend to be a fairly limited class. Hence, we need the notion of the ind-completion of a category. 

\begin{definition} \label{ind}

Let $\C$ be a symmetric tensor category. By $\C^{\ind}$, we denote the ind-completion of $\C$, i.e., the closure of $\C$ under taking filtered colimits of objects in $\C$. 

\end{definition}

The tensor product in $\C$ is exact due to rigidity of $\C$. Hence, it commutes with taking filtered colimits and hence extends to an exact tensor product on $\C^{\ind}$. Additionally, naturality of the braiding implies that the braiding extends to a symmetric structure on $\C^{\ind}$. $\C^{\ind}$ is thus a symmetric $\k$-linear abelian monoidal category in which the tensor product structure $\otimes$ is exact (but it is neither rigid nor locally finite). A specific example of $\C^{\ind}$ that we will repeatedly use in the rest of this paper is the case where $\C$ is  a symmetric fusion category, i.e., when $\C$ is finite and semisimple. In this case, the objects of $\C^{\ind}$ are precisely the (possibly infinite) direct sums of the simple objects in $\C$. 

\begin{remark}

Throughout this paper, we will view the subcategory of $\C$ generated additively by the monoidal unit $\mathbf{1}$ as the category of vector spaces. This gives us a canonical embedding of $\Vec$ inside every symmetric tensor category over $\k$. Objects that are inside this subcategory will be called \emph{trivial objects}.

\end{remark}

\subsection{Algebras, Hopf algebras and modules}

Symmetric tensor categories are naturally equipped with notions of multiplication, associativity, unitality and commutativity. Hence, we can define the notion of an algebra or Hopf algebra fairly naturally inside such a category or its ind-completion, and also examine several important properties such as associativity or unitality.

\begin{definition} Let $(\C, c)$ be a symmetric tensor category over $\k$, with ind-completion $\C^{\ind}$. An associative, unital algebra in $\C^{\ind}$ (also called an ind-algebra in $\C$) is an object $A \in \C^{\ind}$ equipped with multiplication maps $m: A \otimes A \rightarrow A$, $\iota: \mathbf{1} \rightarrow A$ such that the standard commutative diagrams defining associativity and unitality hold. An ind-algebra $A$ is \emph{commutative} if $m \circ c_{A, A} = m.$  If $A, B$ are ind-algebras in $\C$, a morphism $f: A \rightarrow B$ in $\C^{\ind}$ is a \emph{homomorphism} of algebras if $f \circ \iota_{A} = \iota_{B} \text{ and }  f \circ m_{A} = m_{B} \circ (f \otimes f).$

\end{definition}

We can similarly define coalgebras, bialgebras and Hopf algebras.

\begin{definition} Let $\C$ be a symmetric tensor category and $\C^{\ind}$ the ind-completion. A coassociative, counital coalgebra $H$ in $\C^{\ind}$ (also called an ind-coalgebra in $\C$) is object in $\C^{\ind}$ equipped with morphisms $\Delta: H \rightarrow H \otimes H$, $\epsilon : H \rightarrow \mathbf{1}$ that satisfy the standard coassociativity and counitality diagrams. A coalgebra $C$ is \emph{cocommutative} if $c_{C, C} \circ \Delta = \Delta.$

\end{definition}

We can also analogously define a homomorphism of coalgebras.

\begin{definition} A \emph{bialgebra} in $\C^{\ind}$ (also called an ind-bialgebra in $\C$) is an object $B \in \C^{\ind}$ equipped with the structure of both an associative, unital ind-algebra and a coassociative, counital ind-coalgebra such that the comultiplication and counit maps are algebra homomorphisms (or equivalently, the multiplication and unit maps are coalgebra homomorphisms). Here $B \otimes B$ is given the algebra structure by multiplying independently in each tensor component.

\end{definition}

\begin{definition} A Hopf algebra in $\C^{\ind}$ (also called an ind-Hopf algebra) is an object $H \in \C^{\ind}$ equipped with the structure of a bialgebra and an antipode map $S: H \rightarrow H$ that is an isomorphism such that

$$m \circ (S \otimes \id_{H}) \circ \Delta = \iota \circ \epsilon = m \circ (\id_{H} \otimes S) \circ \Delta.$$
If $H$ is a commutative Hopf algebra in $\C^{\ind}$, we will think of it as the \emph{algebra of functions} on an affine group scheme in $\C$. The data of the affine group scheme is equivalent to the data of its commutative Hopf algebra of functions. For affine group schemes $G$ in $\C$, we will use $\O(G)$ to denote the corresponding algebra of functions that defines it.

\end{definition}

For the rest of this section, fix a symmetric tensor category $\C$ over $\k$.

\begin{definition} \label{algebra}

\begin{enumerate}

\item[1.] If $A$ is a commutative ind-algebra in $\C$, then a unital subalgebra $B$ of $A$ is a subobject such that $m(B \otimes B) = B$ and $B$ contains the image of $\iota$.  An ideal $I$ in $A$ is a subobject of $A$ such that $m(A \otimes I) = I$. If $X$ is a subobject, the ideal generated by $X$ is the image $m(A \otimes X)$ under the multiplication map. 

\item[2.] Similarly, if $H$ is a Hopf algebra, then a Hopf subalgebra is a subalgebra $H'$ such that $\Delta(H') \subseteq H' \otimes H'$, and a Hopf ideal is an ideal $I$ such that $\Delta(I) \subseteq H \otimes I \oplus I \otimes H.$

\item[3.] Let $A$ be a commutative ind-algebra in $\C$, with $\C$ semisimple. The \emph{underlying ordinary commutative algebra} is the quotient $\overline{A}:= A/I$, where $I$ is the ideal generated by all simple subobjects of $A$ not isomorphic to $\mathbf{1}$. $\overline{A}$ is an ordinary commutative $\mathbf{k}$-algebra (viewed as an ind-algebra in $\C$ via the canonical inclusion of $\Vec$). A priori this quotient algebra could be $0$, but this does not turn out to be the case for finitely generated algebras in $\Ver_{p}$, due to results proved in \cite{Ven1}.

\item[4.]  If $A$ is a commutative ind-Hopf algebra in $\C$, and $\C$ is semisimple, then $I$ is a proper Hopf ideal. In this case, we call $\overline{A}$ the \emph{underlying ordinary commutative Hopf algebra} associated to $A$.

\item[5.] The invariant subalgebra $A_{0}$ of $A$ is the sum of all the simple subobjects of $A$ isomorphic to $\mathbf{1}$. Note that this is not necessarily a Hopf subalgebra, if $A$ is a Hopf algebra.

\item[6.] Let $H$ be an ind-Hopf algebra in $\C$. The subobject of \emph{primitives} inside $H$, denoted $\Prim(H)$, is the kernel of $\Delta - \id_{H} \otimes \iota - \iota \otimes \id_{H}: H \rightarrow H \otimes H.$

A subobject $X \cong \mathbf{1} \subseteq H$ is \emph{grouplike} if $\Delta(X) = X \otimes X.$

\item[7.] Finally, an important notion is that of a \emph{module}. A left module for an ind-algebra $A$ in $\C$ is an object $M \in \C^{\ind}$ equipped with a map $a: A \otimes M \rightarrow M$
such that the diagrams

$$\begin{tikzpicture}
\matrix (m) [matrix of math nodes,row sep=4em,column sep=4em,minimum width=4em]
{A \otimes A \otimes M & A \otimes M   \\
A \otimes M & M\\};
\path[-stealth]
(m-1-1) edge node[above] {$m \otimes \id_{M}$} (m-1-2)
(m-1-2) edge node[right] {$a$} (m-2-2)
(m-1-1) edge node[left ] {$\id_{A} \otimes a$} (m-2-1)
(m-2-1) edge node[above] {$a$} (m-2-2);
\end{tikzpicture}$$
and

$$\begin{tikzpicture}
\matrix (m) [matrix of math nodes,row sep=4em,column sep=4em,minimum width=4em]
{\mathbf{1} \otimes M & A \otimes M  \\
& M\\};
\path[-stealth]
(m-1-1) edge node[above] {$\iota \otimes \id_{M}$} (m-1-2)
(m-1-2) edge node[left] {$a$} (m-2-2)
(m-1-1) edge node[left=0.2cm] {$\id_{M}$} (m-2-2);
\end{tikzpicture}$$
commute. Note that $A$ is a left module over itself and left ideals are simply left submodules of $A$. If $M, N$ are left A-modules, a \emph{homomorphism} of left $A$-modules from $M$ to $N$ is a morphism $f \in \Hom_{\C^{\ind}}(M, N)$ such that $f \circ a_{M} = a_{N} \circ f.$

We can analogously define right comodules over an ind-coalgebra H via a coaction map $\rho: M \rightarrow M \otimes H$ and define homomorphisms of right comodules.

\end{enumerate}

\end{definition}

The last definition gives the structure of an abelian category to the category of modules over a fixed algebra in $\C^{\ind}$ (or a category of comodules over a fixed coalgebra in $\C^{\ind}$) and these categories naturally come equipped with a faithful, exact functor to $\C^{\ind}$. Moreover, if $H$ is an ind-Hopf algebra then we have the following:

\begin{proposition} There is a natural structure of a tensor category on the category of modules (resp. comodules) over $H$ in $\C^{\ind}$. This category is also equipped with a symmetric structure if $H$ is cocommutative (resp. commutative)

\end{proposition}

\begin{proof} Details can be looked up in \cite{EGNO}. As a brief description of the constructions involved in the proof, the tensor product on the category of modules over $H$ is acquired via the following maps: $a_{M \otimes N}: H \otimes (M \otimes N) \rightarrow M \otimes N$
is just 

$$(a_{M} \otimes a_{N}) \circ (\id_{H} \otimes c_{H, M} \otimes \id_{N}) \circ  (\Delta \otimes \id_{M \otimes N}).$$ \end{proof}

Here are examples of some important algebras, Hopf algebras and modules.

\begin{enumerate} 

\item[1.] Given an object $X \in \C$, the tensor algebra of $X$ is $T(X): = \displaystyle\bigoplus_{n=0}^{\infty} X^{\otimes n}$
with multiplication given by concatenation. This can be given the structure of a Hopf algebra by setting $X$ to be primitive. Taking the graded dual as a Hopf algebra defines the tensor coalgebra $T_{c}(X)$.

\item[2.] With $X$ as above, the symmetric algebra of $X$ is the quotient of the tensor algebra by the ideal generated by $(\id_{X \otimes X} - c_{X, X})(X \otimes X).$ It is a graded quotient Hopf algebra of $T(X)$ with the degree $n$ piece being the coinvariants of $X^{\otimes n}$ under the $S_{n}$-action induced by the braiding.

\item[3.]  With $X$ as above, the exterior algebra of $X$ is the quotient of $T(X)$ by the ideal generated by the kernel of the morphism $\id_{X \otimes X} - c_{X, X}: X \otimes X \rightarrow X \otimes X.$ If $p > 2$, this is the same as the quotient of $T(X)$ by the ideal generated by $(\id_{X \otimes X} + c_{X,X})(X \otimes X).$
\item[4.] If $A$ is any associative unital algebra in $\C^{\ind}$ and $X$ is an object in $\C$, the free left $A$-module generated by $X$ is $A \otimes X \in \C^{\ind}$, with the left action induced by the action of $A$ on itself.

\end{enumerate}

\begin{remark} It is easy to see that $T(X), S(X), A \otimes X$ satisfy the standard universal properties of the free associative algebra, the free commutative algebra and the free $A$-module generated by $X$ respectively (namely that the free functor is left adjoint to the forgetful functor from these categories to $\C^{\ind}$.)

\end{remark}

\subsection{The Verlinde Category $\Ver_{p}$: construction}

The simplest construction of $\Ver_{p}$ is as the semisimplifcation of the category of finite dimensional $\Z/p\Z$ representations over $\k$. Semisimplification of categories is a general process by which we can start with any symmetric tensor category and obtain a semisimple one that is somewhat universal (see \cite{EO1} for details). To define this semisimplifcation process, we need to define the notion of traces. 

\begin{definition} Let $\C$ be a locally finite, rigid, symmetric monoidal additive category in which $\End_{\C}(\mathbf{1}) \cong \k$. If $f: X \rightarrow X$ is a morphism in $\C$, then the \emph{trace} of $f$ is the scalar given by the morphism

$$\begin{tikzpicture}
\matrix (m) [matrix of math nodes,row sep=4em,column sep=4em,minimum width=4em]
{\mathbf{1} & X \otimes X^{*} & X \otimes X^{*} & X^{*} \otimes X & \mathbf{1}  \\};
\path[-stealth]
(m-1-1) edge node[auto] {$\coev_{X}$} (m-1-2)
(m-1-2) edge node[auto] {$f \otimes \id_{X^{*}}$} (m-1-3)
(m-1-3) edge node[auto] {$c_{X, X^{*}}$} (m-1-4)
(m-1-4) edge node[auto] {$\ev_{X}$} (m-1-5); 
\end{tikzpicture}$$
in $\End_{\C}(\mathbf{1}) \cong \k.$ We use $\tr(f)$ to denote the trace of $f$.

\end{definition}

\begin{definition} If $\C$ is a locally finite, rigid, symmetric monoidal additive category as above, then for any $X, Y \in \C$, the space of \emph{negligible} morphisms $\mathcal{N}(X, Y) \subseteq \Hom_{\C}(X, Y)$ consists of those morphisms $f: X \rightarrow Y$ such that for all $g: Y \rightarrow X$, $\tr(g \circ f) = 0$.
The \emph{categorical dimension} of $X \in \C$, denoted $\dim(X)$, is $\tr(\id_{X}).$ We say that $X$ is \emph{negligible} if $\id_{X}$ is a negligible morphism. For indecomposable $X$, this is equivalent to $\dim(X) = 0$.

\end{definition}

\begin{proposition} $\mathcal{N}(X, Y)$ is a tensor ideal.







\end{proposition} 

\begin{proof} See \cite[Lemma 2.3]{EO1}. \end{proof}

\begin{definition} Given a locally finite, rigid, symmetric monoidal additive category $\C$ in which $\End_{\C}(\mathbf{1}) \cong \k$, the quotient category $\overline{\C}$ which has the same objects as $\C$ but in which $\Hom_{\overline{\C}}(X, Y) \cong \Hom_{\C}(X, Y)/\mathcal{N}(X, Y)$
is called the \emph{semisimplification} of $\C$. 

\end{definition}

Here are some important properties of $\overline{\C}$ that can be looked up in \cite{EO1}.

\begin{enumerate}

\item[1.] $\overline{\C}$ is semisimple and hence abelian. Thus, it is a semisimple symmetric tensor category. The monoidal structure on $\overline{\C}$ is induced from that on $\C$ as $\mathcal{N}(X, Y)$ is a tensor ideal.

\item[2.] The simple objects of $\overline{\C}$ are the images under the quotient functor of the indecomposable objects of $\C$ that are not negligible.

\end{enumerate}

\begin{definition} The \emph{Verlinde category} is the semisimplification of the category of finite dimensional $\k$-representations of $\Z/p\Z$. We denote this category as $\Ver_{p}$.

\end{definition}

Let us now describe the tensor structure on $\Ver_{p}$. Proofs of these facts can be looked up in \cite{O, GK, GM}.

\begin{example} \label{Verlindeprop}

\begin{enumerate}

\item[1.] A representation of $\Z/p\Z$ is simply a matrix whose $p$th power is the identity. The indecomposable representations of $\Z/p\Z$ are the indecomposable Jordan blocks of eigenvalue $1$ and size $1$ through $p$. Let us call these representations $M_{1}, \ldots, M_{p}$. The dimension of $M_{i}$ is simply its dimension mod $p$. Hence, $M_{p}$ is the only negligible indecomposable. Thus, the simple objects of $\Ver_{p}$ are: $L_{1}, \ldots, L_{p-1}$, where $L_{i}$ is the image under semisimplification of the indecomposable Jordan block of dimension $i$. 

\item[2.] To describe the monoidal structure, we need to describe the decomposition of $L_{i} \otimes L_{j}$ into direct sum of simples: $L_{i} \otimes L_{j} \cong \displaystyle\bigoplus_{k=1}^{\min(i, j, p-i, p-j)} L_{|j-i| + 2k - 1}.$
This rule is easiest to remember in terms of representations of $SL_{2}(\mathbb{C})$. Let $V_{i}$ be the irreducible representation of $SL_{2}(\mathbb{C})$ of dimension $i$. The decomposition of $L_{i} \otimes L_{j}$ is the same as the decomposition of $V_{i} \otimes V_{j}$ with some irreducibles removed: we remove any representation of dimension $\ge p$ and if $V_{p+r}$ was removed, we also remove $V_{p-r}$.

\end{enumerate}

\end{example}

\begin{definition} For an object $X \in \Ver_{p}^{\ind}$, define $X_{0}$ as the sum of all subobjects of $X$ isomorphic to $\mathbf{1}$ and define $X_{\not=0}$ to be the natural complement of $X_{0}$ in $X$. 

\end{definition}

\subsection{Finitely generated commutative algebras}

Since we are concerned with finitely generated commutative Hopf algebras in this paper, geometric properties of finitely generated commutative algebras will be important to us. In this subsection, we list some important definitions and results from \cite{Ven1} for the convenience of the reader. Fix a symmetric tensor category $\C$ over $\k$. 

\begin{definition} We say that $\C$ admits a \emph{Verlinde fiber functor} if there exists a tensor functor $F: \C \rightarrow \Ver_{p}$. This holds for $\C = \Ver_{p}$ in particular. \end{definition}

\begin{definition} We say that a commutative algebra $A$ in $\C^{\ind}$ is \emph{finitely generated} if there exists some object $X \in \C$ and a surjective homomorphism of algebras $S(X) \rightarrow A.$
For an arbitrary commutative algebra $A \in \C^{\ind}$, we say that an $A$-module $M \in \C^{\ind}$ is finitely generated if there exists an object $X \in \C$ and a surjective homomorphism of $A$-modules $A \otimes X \rightarrow M.$ For an affine group scheme $G$ in $\C$, with function algebra $\O(G)$, we say that $G$ is of \emph{finite type} if $\O(G)$ is finitely generated.

\end{definition}

\begin{definition} \label{Noetherian} For a commutative ind-algebra $A$, we say that an $A$-module $M$ is \textit{Noetherian} if its submodules satisfy the ascending chain condition, i.e., that for any sequence of submodules $M_{0} \rightarrow M_{1} \rightarrow M_{2} \rightarrow \cdots$
in which the morphisms are monomorphisms, there exists some $n$ such that for all $N \ge n$, the map $M_{N} \rightarrow M_{N+1}$ is an isomorphism. We say that $A$ is a \textit{Noetherian algebra} if all of its finitely generated modules are Noetherian.

\end{definition}

This is equivalent to finite generation of submodules (see \cite{Ven1} for a proof). We also have a definition of an invariant subalgebra. 

\begin{definition} Let $A$ be a commutative algebra in $\C^{\ind}$. The invariant subalgebra $A_{0}$ is the sum of all simple subobjects of $A$ isomorphic to $\mathbf{1}$ and $A_{\not=0}$ is the sum of all simple subobjects not isomorphic to $\mathbf{1}$.

\end{definition}

In $\Ver_{p}$, and categories fibered above it, the following results from \cite{Ven1} strongly control the geometry of finitely generated commutative algebras, more or less reducing it to ordinary commutative algebra.

\begin{lemma} \label{nilpotence} For $i < p$, let $L_{i}$ be the simple object in $\Ver_{p}$ corresponding to the indecomposable representation of $\Z/p\Z$ of dimension $i
$. For $i > 1$, $S^{N}(L_{i}) = 0 \text{ for } N > p -i.$

\end{lemma}

\begin{corollary} If $A$ is a finitely generated commutative ind-algebra in $\Ver_{p}$, then the ideal generated by $A_{\not=0}$ is nilpotent. 

\end{corollary}

\begin{theorem} \label{Hilb2} Let $\C$ be a symmetric tensor category over $\k$. If $\C$ admits a Verlinde fiber functor, then every finitely generated commutative ind-algebra $A \in \C^{\ind}$ is Noetherian.
\end{theorem}

\begin{theorem} \label{Inv} Suppose $\C$ is a symmetric finite tensor category over $\k$ that admits a Verlinde fiber functor $F : \C \rightarrow \Ver_{p}$. Let $A \in \C^{\ind}$ be a finitely generated commutative algebra and let $A_{0}$ be its invariant subalgebra. Then, $A_{0}$ is finitely generated and $A$ is a finitely generated $A_{0}$-module.
\end{theorem}

\begin{remark} These results, in particular the lemma above, essentially allow us to reduce geometry of finitely generated commutative algebras in $\Ver_{p}^{\ind}$ to the geometry of $A_{0}$ and $A/(A_{\not=0})$, which are finitely generated ordinary commutative algebras over $\k$.

\end{remark}

\section{\Large{\textbf{More commutative algebra in $\Ver_{p}$}}}

In this section, we use the above results to extend some classical commutative algebra theorems to the setting of $\Ver_{p}$. Most of the results of this section are going to be modified versions of results from \cite{Mat}. Themain tool that allows us to prove these theorems is to examine finitely generated commutative ind-algebras $A$ in $\Ver_{p}$ as modules over $A_{0}$ and nilpotent thickenings of $\overline{A}$ to then reduce proofs to the analgous proofs in \cite{Mat}.

We will also want some of the results in this section to hold for completions of algebras in $\Ver_{p}$. These are not ind-algebras but are rather pro-algebras. So we begin with a technical definition.

\begin{definition} $\Ver_{p}^{\pro}$ is the closure of $\Ver_{p}$ under projective limits.

\end{definition}

\begin{remark}

Objects in $\Ver_{p}^{\pro}$ are projective limits of objects in $\Ver_{p}.$ Since $\Ver_{p}$ is semisimple, these are just possibly infinite products of simple objects in $\Ver_{p}$. This is a monoidal abelian category under the completed tensor product, defined in exactly the same formal manner as in the category of vector spaces. Full duals for objects do not exist, but we can define a  dualization functor $\Ver_{p}^{\ind} \rightarrow \Ver_{p}^{\pro}$, by dualizing the inductive system to get a projective system.

\end{remark}



\begin{definition} Let $A$ be a associative, unital algebra in $\Ver_{p}^{\pro}$ or $\Ver_{p}^{\ind}$. The \emph{Jacobson radical} $\mathcal{J}$ of $A$ is the intersection of $\Ann(M)$ over all simple left $A$-modules $M$, where $\Ann(M)$ is the largest subobject in $A$ that acts as $0$ on $M$.

\end{definition}

Since the annihilator of a module is a two sided ideal, the Jacobson radical of an algebra is also a two sided ideal. 

\begin{proposition}[Nakayama Lemma] \label{Nakayama} Let $A$ be a commutative Noetherian algebra in $\Ver_{p}^{\pro}$ or $\Ver_{p}^{\ind}$. Let $M$ be a finitely generated $A$-module and let $I$ be any ideal contained in the Jacobson radical of $A$. If $IM = M$, then $M = 0$. 

\end{proposition}

\begin{proof} By Noetherianity of $A$,  if $M$ is not zero, we can apply Zorn's Lemma to the set of proper submodules of $M$ to show the existence of a maximal proper submodule $N$ of $M$. Since $M/N$ is simple, $I(M/N) = 0$. Hence, $IM \subseteq N$, a contradiction. \end{proof}

We will want a better description of the Jacobson radical of commutative rings in $\Ver_{p}$. 

\begin{proposition} \label{Jacobson} Let $A$ be a commutative ind-algebra in $\Ver_{p}$. Then, the Jacobson radical of $A$ is the intersection of all maximal ideals of $A$.

\end{proposition}

\begin{proof} Note that the Jacobson radical $\mathcal{J}$ must be contained inside every maximal ideal since $A/\m$ is a simple $A$-module for every maximal ideal $\m$. Hence, the Jacobson radical is contained inside the intersection of all maximal ideals. To show the reverse inclusion, let $I$ be the ideal generated by all simple subobjects not isomorphic to $\mathbf{1}$. Then, $I$ is a locally nilpotent ideal by Lemma \ref{nilpotence}. Hence, if $M$ is any simple $A$-module, $IM \not=M$ by local nilpotence of $I$ and hence $IM = 0$. Thus, every simple $A$-module is a simple $A/I$-module and $I \subseteq \mathcal{J}$. 

Now, the Jacobson radical of $A/I$ is the intersection of all maximal ideals of $A/I$, as this is an ordinary commutative algebra over $\k$. But the Jacobson radical of $A/I$ must be contained inside $\mathcal{J}/I$ as every simple $A$-module is also a simple $A/I$-module. Hence, $\mathcal{J}/I$ is contained in the intersection of all $\m/I$, with $\m$ ranging over all maximal ideals of $A$. Thus, $\mathcal{J} \subseteq \bigcap \m + I$, but as $I \subseteq \mathcal{J}$, this implies $\mathcal{J} \subseteq \bigcap \m.$
\end{proof}

We will also want a good description of the nilradical of a commutative ring in $\Ver_{p}$.

\begin{definition} Let $A$ be a commutative ind-algebra in $\Ver_{p}$. The nilradical $N$ of $A$ is the sum of all nilpotent ideals in $A$.

\end{definition}

\begin{remark} The nilradical is always a locally nilpotent ideal, so if $A$ is Noetherian, then it is nilpotent. In particular, this holds if $A$ is finitely generated. 

\end{remark}

Note that by Lemma \ref{nilpotence}, $N$ must contain all the simple subobjects of $A$ not isomorphic to $\mathbf{1}$. 

\begin{proposition} The nilradical of a finitely generated commutative ind-algebra in $\Ver_{p}$ is the intersection of all the maximal ideals. Hence, it coincides with the Jacobson radical.

\end{proposition}

\begin{proof} Clearly every maximal ideal contains the nilradical. Hence, the nilradical is always contained inside the intersection of all maximal ideals. To show the reverse, we note that reducing mod $N$ gives us a finitely generated reduced commutative algebra over $\k$. Such algebras are Jacobson rings and hence, mod $N$, the intersection of all maximal ideals is $0$, since it is the intersection of all the prime ideals. 
\end{proof} 

\begin{definition}[Rees Algebra] Let $A$ be a commutative ind-algebra in $\Ver_{p}$ and $I$ be an ideal of $A$. The Rees algebra of $A$ with respect to $I$ is the blowup $\Bl_{I}A := \displaystyle\bigoplus_{n=0}^{\infty} I^{n}.$
If $M$ is an $A$-module with a descending filtration $F = \{M_{n}\}$ such that $IM_{n} \subseteq M_{n+1}$,  then $\Bl_{F} M := \displaystyle\bigoplus_{n=0}^{\infty} M_{n}.$
In particular, $\Bl_{I}M$ is the blowup with respect to the filtration $I^{n}M$.

\end{definition}

\begin{remark} If $A$ is finitely generated, so are $I$ and $\Bl_{I}A$, and hence $\Bl_{I}A$ is Noetherian.

\end{remark}

\begin{proposition}[Artin-Rees Lemma] \label{Artin-Rees} Let $A$ be a finitely generated commutative ind-algebra in $\Ver_{p}$.  Let $I$ be an ideal in $A$ and suppose $M$ is a finitely generated $A$-module. Let $N$ be a submodule of $M$. Then, there exists some $k \ge 1$ such that for $n \ge k$, $I^{n}M \cap N = I^{n-k}((I^{k}M)\cap N).$

\end{proposition}

\begin{proof} We first prove that for any descending filtration $M_{n}$ on $M$, $\Bl_{F} M$ is finitely generated over $\Bl_{I}A$ if and only if $IM_{n} = M_{n+1}$ for $n \gg 0$. If $IM_{n} = M_{n+1}$ for $n > n_{0}$, then $\Bl_{F} M$ is generated by $M_{0}, \ldots, M_{n_{0}}$. Hence, it is finitely generated. Conversely, if it is finitely generated, say by $M_{0}, \ldots, M_{n_{0}}$ for some $n_{0}$, then for any $n \ge n_{0}$

$$M_{n} \subseteq \bigoplus_{j=n-n_{0}}^{n} I^{j} \cdot M_{n-j}.$$
Since $I^{j}M_{n-j} \subseteq IM_{n-1}$, $M_{n} \subseteq IM_{n-1}$ and the reverse inclusion is by definition of the filtration.

Now, consider the natural $I$-filtration on $M$, where $M_{n} = I^{n}M$ and the induced filtration $F$ on $N$, with $N_{n} = (I^{n}M \cap N).$ By the result proved above, $\Bl_{I}M$ is finitely generated. Hence, it is Noetherian by Theorem \ref{Hilb2} and $\Bl_{F}N$ is a finitely generated module over $\Bl_{I} A$. Thus there exists some $k$ such that for $n > k$, $I N_{n-1} = N_{n}$. This proves the proposition.
\end{proof}

\begin{proposition}[Krull Intersection Theorem] \label{Krull} Let $A$ be a finitely generated local commutative ind-algebra in $\Ver_{p}$. Let $I$ be a proper ideal in $A$. Then $\bigcap_{n} I^{n} = 0.$ 
More generally, this holds when $A$ is local Noetherian, the Artin-Rees lemma holds for $A$ and its maximal ideal and the Jacobson radical of $A$ is the maximal ideal.

\end{proposition}

\begin{proof} This is an immediate consequence of the Artin-Rees Lemma and Nakayama's Lemma, noting that Proposition \ref{Jacobson} implies that the Jacobson radical of $A$ is the unique maximal ideal. 
\end{proof}

\subsection{Localizations}

In this subsection, we want to define localizations of finitely generated commutative algebras in $\Ver_{p}$. Defining localizations categorically is a bit of a pain so we instead take a shortcut. If $A$ is a commutative ind-algebra in $\Ver_{p}$, then, as an object in $\Ver_{p}^{\ind}$, $A \cong A_{0} \oplus A_{\not=0}$
with $A_{\not=0}$ the direct sum of all simple subobjects in $A$ not isomorphic to $\mathbf{1}$.

\begin{definition} Let $A$ be commutative ind-algebra in $\Ver_{p}$. A \emph{multiplicative subset} $S$ of $A$ is a multiplicative subset of $A_{0}$.

\end{definition}

We can use elements to denote a multiplicative subset, since $A_{0}$ is just some vector space. Hence, we can localize with respect to $S$ in exactly the same manner as in ordinary commutative algebra. 

\begin{definition} Given a multiplicative subset $S = \{x_{i} : i \in I\}$ of a commutative ind-algebra $A$ (with $I$ some index set), the \emph{localization} $A_{S}$ of $A$ with respect to $S$ is the ind-algebra

$$\left(A \otimes \mathbf{k}[x_{i}^{-1}: i \in I]\right)/(x_{i}x_{i}^{-1} - 1).$$

\end{definition}

Put in simpler terms, we can treat a subobject of $A$ isomorphic to $\mathbf{1}$ as the line spanned by an actual element in $A$, and we can manually adjoin inverses to these elements in order to localize.

\begin{proposition} If $\m$ is a maximal ideal of $A$, then the elements of $A_{0}$ that aren't in $\m$ form a multiplicative subset. Hence, we can define an algebra $A_{\m}$, the localization of $A$ at $\m$, as the localization with respect to this multiplicative subset.

\end{proposition}

\begin{proof} The proof is obvious.
\end{proof}

The standard facts about localizations at maximal ideals still hold for finitely generated algebras.

\begin{proposition} \label{localizeprop} Let $A$ be a finitely generated commutative ind-algebra in $\Ver_{p}$ and let $\m$ be a maximal ideal in $A$.

\begin{enumerate}

\item[1.] There is a natural map from $A$ into $A_{\m}$ whose kernel is the sum of annihilators of elements of $S$ under the multiplication action of $A$ on itself.

\item[2.] If $K$ is the kernel of the map from $A$ into $A_{\m}$, ideals in $A_{\m}$ correspond to ideals of $A/K$ contained in $\m/K$. Hence, $A_{\m}$ is local with maximal ideal $\m$.

\item[3.] Finitely generated $A_{\m}$ modules are Noetherian.

\item[4.] Finitely generated $A_{\m}$ algebras are Noetherian as algebras, i.e., their finitely generated modules are Noetherian. 

\item[5.] The Artin-Rees lemma and Krull Intersection Theorem hold for $A_{\m}$.

\item[6.] Define $\m^{\infty}$ as $\cap_{n} \m^{n}$. Then, $\m(\m^{\infty}) = \m^{\infty}.$

\end{enumerate}

\end{proposition}

\begin{proof} 

\begin{enumerate}

\item[1.] The proof of 1 is obvious. 

\item[2.] To prove 2, let $I$ be an ideal in $A$ not contained in $\m$. Then, $I = I_{0} \oplus I^{\not=0}$
with $I_{\not=0}$ nilpotent and hence contained in $\m$ by Lemma \ref{nilpotence} and $I_{0}$ is an ideal in $A_{0}$. Now, $I_{0}$ is not contained in $\m_{0}$, and hence contains a unit under localization. Thus, $I_{\m} = A_{\m}$. Taking the pre-image under the map from $A$ to $A_{\m}$ establishes the correspondence.

\item[3.] To prove 3, we just need to prove Noetherianity for free modules, but these are just localizations of free $A$-modules.

\item[4.] As in the unlocalized case, it suffices to prove the statement for algebras of the form $A_{\m} \otimes S(X)$ for some object $X \in \Ver_{p}$, and the statement follows from Noetherianity of $A \otimes S(X)$ (\cite{Ven1}).

\item[5.] All we need is Noetherianity of finitely generated algebras over $A_{\m}$ for Artin-Rees and Nakayama, and Krull Intersection follows from Artin-Rees and Nakayama.

\item[6.] By the Krull-Intersection theorem applied to $A_{\m}$ we see that there is an element $a \in A_{0}$ not in $\m \cap A_{0}$ such that $a\m^{\infty} = 0$. Let $\lambda \in \k^{*}$ be the projection of $a$ in $A/\m \cong \mathbf{1} \cong \k$. Then, $a - \lambda \in \m \cap A_{0}$. Hence, $a - \lambda$ acts as $-\lambda$ on $\m^{\infty}$, which is hence $\m$-stable.

\end{enumerate}
\end{proof}

We end this section with a useful proposition regarding completions of commutative algebras at maximal ideals. 

\begin{definition} Let $A$ be a finitely generated commutative ind-algebra in $\Ver_{p}$ and $\m$ a maximal ideal of $A$. The \emph{completion} of $A$ at $\m$ is the inverse limit $\widehat{A_{\m}} := \varprojlim_{n} A/\m^{n}$, a commutative algebra in $\Ver_{p}^{\pro}$. 

\end{definition}

Using the embedding of $\Ver_{p}^{\pro}$ into $\Ver_{p}^{\ind}$, we can view $\widehat{A_{\m}}$ as a commutative algebra in $\Ver_{p}^{\ind}$. There is a natural map from $A$ into its completion whose kernel is $\cap_{n} \m^{n}$. 

\begin{corollary} \label{injectcomplete} The product map $A \rightarrow \displaystyle\prod_{\m \text{ maximal ideal in } A} \widehat{A}_{\m}$
is an injection.

\end{corollary}

\begin{proof} This map factors through the natural map $A \rightarrow \prod_{\m} A_{\m}$ into the localizations at every maximal ideal. It suffices to prove that this map is injective because the map $A_{\m} \rightarrow \widehat{A_{\m}}$ is injective by part 6 of the proposition above. But by definition $\prod_{\m} A_{\m} = \prod_{\m_{0}} A_{\m_{0}}$, the localization of $A$ with respect to the maximal ideals of $A_{0}$. Since $A$ is a finitely generated module over Noetherian $A_{0}$ by Theorem \ref{Inv}, the proposition follows from classical commutative algebra. 
\end{proof}

\section{\Large{\textbf{Cocommutative coalgebras in $\Ver_{p}$}}}

For the rest of this paper, we assume $p > 3$, because for $p = 2$, $\Ver_{p} = \Vec$ and for $p = 3$, $\Ver_{p} = \sVec$ and everything we have to say is known in these cases.

\subsection{Pairings}

For $X \in \Ver_{p}^{\ind}$, we have an evaluation map $X^{*} \otimes X \rightarrow \mathbf{1}$. We want to use the language of pairings to study this evaluation map further. 

\begin{definition}

A \emph{pairing} of objects $V, W$ in $\Ver_{p}^{\ind}$ or $\Ver_{p}^{\pro}$ is a map $V\otimes W \rightarrow \mathbf{1}.$
The \emph{left kernel} of the pairing is the biggest subobject $V' \subseteq V$ such that the pairing restricted to $V' \otimes W$ is $0$. The \emph{right kernel} is defined analagously in $W$. The pairing is said to be \emph{non-degenerate} if both the left and right kernels are $0$.

\end{definition}

\begin{example} If $X \in \Ver_{p}$, then the evaluation pairing $X^{*} \otimes X \rightarrow \mathbf{1}$ is non-degenerate. This is easily seen from the diagrams defining the compatibility properties between evaluation and coevaluation. If $X \in \Ver_{p}^{\ind}$, then decomposing $X$ into a direct sum of simple objects in $\Ver_{p}$ allows us to extend this statement to the pairing between $X^{*}$ and $X$ in $\Ver_{p}^{\ind}$ as well. Here, we use the embedding of $\Ver_{p}^{\pro}$ inside $\Ver_{p}^{\ind}$ to identify $X^{*}$ as an object in $\Ver_{p}^{\ind}$.

\end{example}

\begin{definition} Let $X \in \Ver_{p}^{\ind}$. We say that $Y \subseteq X^{*}$ is \emph{dense} if the evaluation pairing restricted to $Y \otimes X$ is non-degenerate. 

\end{definition}

\begin{remark} Note that if $X \in \Ver_{p}$, then a non-degenerate pairing between $Y$ and $X$ induces an isomorphism $Y \cong X^{*}$. This does not have to be the case if $X \in \Ver_{p}^{\ind}$ of infinite length. In this case, a non-degenerate pairing between $Y$ and $X$ only induces an injection $Y \rightarrow X^{*}$ with dense image.

\end{remark} 

\begin{definition} Let $X, Y \in \Ver_{p}^{\ind}$ and fix a pairing $b: Y \otimes X \rightarrow \mathbf{1}.$
For $W \subseteq X$, we define the \emph{complement} $W^{\perp} \susbeteq Y$ as the biggest subobject $W'$ of $Y$ such that $b|_{W' \otimes W} = 0.$ For $W \subseteq Y$, define the complement $W^{\perp}$ in $X$ analogously.

\end{definition} 

\begin{proposition} \label{dualident} Fix a non-degenerate pairing $\eta: Y \otimes X \rightarrow \mathbf{1}$ in $\Ver_{p}^{\ind}$. Let $W$ be a subobject of $X$, then $\eta$ descends to a non-degenerate pairing $W^{\perp} \otimes X/W \rightarrow \mathbf{1}.$

\end{proposition}

This proposition is obvious from the definition of orthocomplements. If $X, Y \in \Ver_{p}$, then what this proposition allows us to do is identify $(X/W)^{*}$ as the subobject of $Y$ that kills $W$ under the non-degenerate pairing. In particular, we can apply this to the evaluation pairing between an object and its dual.

Given pairings between Hopf algebras, we will want some additional compatibility.

\begin{definition} Let $H, A$ be ind-Hopf algebras in $\Ver_{p}$. Let us use $m_{H}, m_{A}$ to denote multiplication, $\Delta_{H}, \Delta_{A}$ to denote comultiplication, $\iota_{H}, \iota_{A}$ to denote the unit and $\epsilon_{H}, \epsilon_{A}$ to denote the counit maps. A pairing $b: H \otimes A \rightarrow \mathbf{1}$ between ind-Hopf algebras in $\Ver_{p}$ is said to be a \emph{Hopf pairing} if the following hold

\begin{enumerate}

\item[1.]  $b \circ (m_{H} \otimes \id_{A}) =  (b \otimes b) \circ (\id_{H} \otimes c_{H, A} \otimes \id_{A}) \circ (\id_{H \otimes H} \otimes \Delta_{A})$ as maps from $H \otimes H \otimes A \rightarrow \mathbf{1}.$

\vspace{0.2cm}

\item[2.] $b \circ (\id_{H} \otimes m_{A}) = (b \otimes b) \circ (\id_{H} \otimes c_{H, A} \otimes \id_{A}) \circ (\Delta_{H} \otimes \id_{A \otimes A})$ as maps from $H \otimes A \otimes A \rightarrow \mathbf{1}$. 

\vspace{0.2cm}

\item[3.] $b \circ (\iota_{H} \otimes \id_{A}) = \epsilon_{A}$ as maps from $A \rightarrow \mathbf{1}$.

\vspace{0.2cm}

\item[4.] The same with $H$ replacing $A$.

\end{enumerate}

If $A$ and $H$ are $\mathbb{Z}_{\ge 0}$-graded, the pairing is said to be graded if $b|_{H(n) \otimes A(m)} = 0$ unless $n = m$.

\end{definition}

This language of Hopf pairings will prove useful when we later construct the dual coalgebra to a commutative Hopf algebra.

\subsection{Finiteness property of coalgebras}

In the rest of this section, we want to state some facts about the structure theory of cocommutative coalgebras in $\Ver_{p}^{\ind}$. These facts will largely follow from dualization to the setting of finitely generated commutative algebras in $\Ver_{p}^{\ind}$. For this section, fix a cocommutative coalgebra $C$ in $\Ver_{p}^{\ind}$. Let $C_{0}$ be the isotypic component of $C$ corresponding to $\mathbf{1}$. Let $ J = \Delta^{-1}(C_{0} \otimes C_{0})$. $J$ is a cocommutative coalgebra over $\k$.

\begin{proposition} \label{coalg-finite-length} $C$ is the sum of cocommutative subcoalgebras in $\Ver_{p}$ (i.e. those of finite length).

\end{proposition} 

\begin{proof} The proof of this is standard. If $X$ is a subobject of $C$ of finite length, simply take the sum of all tensor factors appearing inside $\Delta(X)$. This sum is an object in $\Ver_{p}$, and coassociativity shows that it is closed under $\Delta$. \end{proof} 

\subsection{Coradical of a cocommutative coalgebra and irreducibility}

\begin{definition} We say that $C'$ is a \emph{subcoalgebra} of $C$ if $\Delta(C') \subseteq C' \otimes C'.$ Given a right comodule $M$ for $C$, we say that $M'$ is a \emph{subcomodule} of $M$ is $\Delta(M') \subseteq M' \otimes C.$

\end{definition}

\begin{definition} We say that a coalgebra $C$ is simple if it has no subcoalgebras. We say that a comodule $M$ is simple if it has no subcomodules.

\end{definition} 

\begin{definition} Define the \emph{coradical} of $C$, denoted $\Corad(C)$, as the sum of all simple subcoalgebras of $C$. Note that cocommutativity implies that subcomodules are actually subcoalgebras.

\end{definition}

\begin{definition} Let $M$ be a $C$-comodule. The \emph{cosocle} of $M$ is the maximal semisimple quotient comodule of $M$.

\end{definition}

\begin{proposition} \label{corad-0} $\Corad(C) \subseteq J.$

\end{proposition}

\begin{proof} Using Proposition \ref{coalg-finite-length} , we can reduce to the case of $C$ being finite length. In this case, $C^{*}$ is a commutative algebra in $\Ver_{p}$ (of finite length), and simple subcomodules of $C$ correspond to simple quotient modules of $C^{*}$. But simple quotients are all of the form $C^{*}/\m \cong \mathbf{1}$, with $\m$ a maximal ideal of $C^{*}$, as any maximal ideal of $C^{*}$ contains the ideal generated by $C^{*}_{\not=0}$ by Lemma \ref{nilpotence}. Hence, any simple subcomodule $C'$ of $C$ is isomorphic to $\mathbf{1}$ as objects in $\Ver_{p}$. Since these are subcomodules and $C$ is cocommutative, it is clear that $\Delta(C') \subseteq C' \otimes C'$ and hence $C' \subseteq J.$  
\end{proof}

\begin{corollary} $\Corad(C)$ is the span of grouplike elements in $J$. Hence, every simple subcomodule of $C$ is isomorphic to $\mathbf{1}$ as an object in $\Ver_{p}$.

\end{corollary}

\begin{definition} We say that $C$ is \emph{irreducible} if $\Corad(C) \cong \mathbf{1}$ in $\Ver_{p}$. This is equivalent to $J$ being irreducible, i.e., having only one grouplike element.

\end{definition}

\begin{remark} We use the term irreducible here to stay consistent with the terminology in \cite{M1}. In other sources, such coalgebras are often called connected or coconnected instead.

\end{remark}

\begin{proposition} \label{irreducible-local} If $C$ is a cocommutative coalgebra in $\Ver_{p}$, then $C$ is irreducible if and only if $C^{*}$ is local.

\end{proposition} 

\begin{proof} The proof of this is immediate. $C$ being irreducible means it has only 1 simple subcoalgebra, which is equivalent to $C^{*}$ having only one simple quotient, which is equivalent to $C^{*}$ having a unique maximal ideal.
\end{proof}

\begin{definition} For a grouplike element $g$ in $J \subseteq C$, define the irreducible component of $C$ containing $g$, denoted $C_{g}$, as the maximal irreducible subcoalgebra of $C$ containing $g$. 

\end{definition}

This irreducible component exists because if $C', C''$ are two irreducible subcoalgebras of $C$ containing a grouplike element $g$, then $C' + C''$ is also an irreducible subcoalgebra containing $g$. Moreover, we have the following reducibility statement.

\begin{proposition} \label{complete-reducibility} If $C$ is a cocommutative coalgebra in $\Ver_{p}^{\ind}$ and $G(C)$ is the subset of grouplike elements in $C$, then $C \cong \displaystyle\bigoplus_{g \in G(C)} C_{g}$
as a coalgebra.

\end{proposition}

Via reduction to coalgebras of finite length in $\Ver_{p}$ and Proposition \ref{irreducible-local}, this proposition reduces to the following statement in commutative algebra in $\Ver_{p}$.

\begin{proposition} Let $A$ be a commutative algebra in $\Ver_{p}$. Let $S$ be the complete set of primitive idempotents in $A_{0}$. Then, $A \cong \prod_{e \in S} Ae$
with $Ae$ a local commutative algebra in $\Ver_{p}$.

\end{proposition} 

\begin{proof} It is clear that given such a set of primitive idempotents, we get a direct product decomposition. What we need to show to prove the proposition is that $Ae$ is local. 

Since $A$ has finite length, it is in particular, finitely generated. Hence, the ideal $I$ generated by $A_{\not=0}$ is nilpotent. Thus, the map $A_{0} \rightarrow A/I$ is a surjection with nilpotent kernel, and the idempotents of both algebras correspond. Let $\overline{e}$ be the image of $e \in S$ under this surjection. Then, the set $\overline{S} = \{\overline{e} : e \in S\}$ forms a set of primitive idempotents in $A/I$. Thus, $Ae/Ie = \overline{A/I}e$ is local, and since $Ie$ is a nilpotent ideal, $Ae$ is local as well. \end{proof}

\subsection{Coradical filtration}

We define the coradical filtration $C(i)$ on $C$ inductively.

\begin{definition} $C(0) = \Corad(C)$. $C(n)$ is the largest subobject of $C$ such that 

$$\Delta(C(n)) \subseteq C(n-1) \otimes C + C \otimes C(0).$$

\end{definition}

The following proposition is standard in the case when $C$ is a cocommutative ind-coalgebra in $\Vec.$ Proofs can be found in \cite[Chapter 9]{S}. The proof of the proposition for $C$ a cocommutative ind-coalgebra in $\Ver_{p}$ carries over without change. 

\begin{proposition} \label{coradical-filtration}

\begin{enumerate}

\item[1.] $C(n)$ is a subcoalgebra of $C$ with $\Delta(C(n)) \subseteq \displaystyle\sum_{i=0}^{n} C(i) \otimes C(n-i).$

\item[2.] $C(i) \subseteq C(i+1)$. 

\item[3.] $\bigcup_{i=0}^{\infty} C(i) = C.$ 

\item[4.] If $f: C \rightarrow C'$ is a homomorphism of coalgebras such that $f(C_{0}) \subseteq C'_{0}$, then $f(C(i)) \subseteq C'(i)$. This is always true if $C$ is irreducible. 

\item[5.] If $C$ is a cocommutative ind-Hopf algebra in $\Ver_{p}$ with multiplication $m$ and antipode $S$, then $m(C(i) \otimes C(j)) \subseteq C(i+j)$
and $S(C(i)) \subseteq C(i).$

\end{enumerate}

\end{proposition}

For each grouplike element in a cocommutative coalgebra, we can define a space of primitives. 

\begin{definition} For $g \in G(C)$, let $i_{g}: \mathbf{1} \rightarrow C$ be the inclusion of $g$ into $J \subseteq C$. Define the $g$-primitives as

$$\Prim_{g}(C) = \ker(\Delta - i_{g} \otimes \id_{C} - \id_{C}\otimes i_{g}).$$
If $C$ is irreducible, define $\Prim(C)$ as the space of primitives with respect to the unique grouplike element in $C$.

\end{definition} 

This definition agrees with our definition of primitives for cocommutative Hopf algebras in $\Ver_{p}^{\ind}$ if the Hopf algebra is irreducible. 

For the rest of this section, assume $C$ is irreducible in addition to being cocommutative, let $g$ be its unique grouplike element and let $i_{g}$ be the inclusion of $g$ into $C$. We want to analyze the coradical filtration on $C$ a little more.

\begin{proposition} \label{corad-prop}

\begin{enumerate}

\item[1.] $C(0) = \k g$.

\item[2.] $\Prim(C) \subseteq \ker(\epsilon)$, with $\epsilon : C \rightarrow \mathbf{1}$ the counit.

\item[3.] $C(1) = C(0) \oplus \Prim(C).$

\item[4.] Define $C(i)^{+} = C(i) \cap \ker(\epsilon)$. Let $\Delta$ denote the comultiplication map. Then, $C(i) = C(0) \oplus  C(i)^{+}$ and

$$(\Delta - i_{g} \otimes \id_{C} - \id_{C} \otimes i_{g})(C(n)^{+}) \subseteq C(n-1)^{+} \otimes C(n-1)^{+}.$$

\item[5.] Let $C, D$ be irreducible cocommutative coalgebras in $\Ver_{p}^{\ind}.$ Then, the coradical filtration on $C \otimes D$ is the tensor product on the coradical filtrations on $C$ and $D$ respectively:

$$(C \otimes D)(n) = \displaystyle\sum_{i=0}^{n} C(i) \otimes D(n-i).$$

\end{enumerate}

\end{proposition}  

\begin{proof} Statement $1$ follows from definition of irreducibility. Statement $2$ follows from the counit axiom. The proof of Statement 3 is essentially the proof of \cite[Proposition 10.0.1]{S}, which we restate here categorically. Note that statement 2 and the fact that $\epsilon(g) = 1$ implies that $C(0) + \Prim(C) $ is a direct sum decomposition, hence we just need to show that  $C(1) = C(0) + \Prim(C).$
Let $C(1) = C(1)_{0} \oplus C(1)_{\not=0}$ be the decomposition of $C(1)$ into the isotypic component corresponding to $\mathbf{1}$ and the natural complement. By definition of $C(1)$,

$$\Delta(C(1)_{\not=0}) \subseteq C(0) \otimes C(1)_{\not=0}  \oplus C(1)_{\not=0} \otimes C(0).$$
Hence, we have a map 

$$\Delta - i_{\g} \otimes \id - \id \otimes i_{\g} : C(1)_{\not=0} \rightarrow C(0) \otimes C(1)_{\not=0} \oplus C(1)_{\not=0} \otimes C(0).$$
Since $\epsilon: C(0) \rightarrow \mathbf{1}$ is an isomorphism, this map is determined by its compositions with $\epsilon \otimes \id_{C}$ and $\id_{C} \otimes \epsilon$. By the counit axiom, and the fact that $\epsilon$ kills $C_{\not=0}$, these compositions are both $0$. Hence, 

$$(\Delta - i_{g} \otimes \id - \id \otimes i_{g}) (C(1)_{\not=0}) = 0.$$
Thus, we just need to show that $C(1)_{0} = C(0) + \Prim(C)_{0}$, and this follows in exactly the same manner as in the proof of \cite[Proposition 10.0.1]{S}, since these are just vector spaces. 

Statement 4 is just Proposition 10.0.2 in \cite{S}. This proposition relies on results analogous to the ones we stated in Proposition \ref{coradical-filtration}. With these results, the proof there works without change, since everything can be stated in terms of the map $\Delta - i_{g} \otimes \id - \id \otimes i_{g}$ and be made element free.

Similarly, statement 5 is \cite[Corollary 11.0.6]{S}, the proof of which uses the fact that $C^{*}$ and $D^{*}$ are topological algebras in $\Ver_{p}^{\pro}$ (under the completed tensor product) but is otherwise element free. 
\end{proof}

\begin{corollary} Let $C'$ be any coalgebra in $\Ver_{p}^{\ind}$ and let $f: C \rightarrow C'$ be a coalgebra map. Then, 

\begin{center} $f$ is injective $\Leftrightarrow f|_{\Prim(C)}$ is injective \end{center}

\end{corollary}

\begin{proof} This is \cite[Lemma 11.0.1]{S}. The proof requires no change to be made suitable for $\Ver_{p}$. 
\end{proof} 

\begin{corollary} \label{coideal-primitives} If $I$ is a coideal in $C$, then $I \cap \Prim(C) = 0 \Rightarrow I = 0.$

\end{corollary} 

\begin{definition} Given a cocommutative coalgebra $C$ in $\Ver_{p}^{\ind}$, let $C_{\gr}$ be the associated graded coalgebra under the coradical filtration. \end{definition}

An important property of this filtration that we will use is the following proposition, whose proof is standard.

\begin{lemma} \label{corad-grading} 

\begin{enumerate}

\item[1.] $C_{\gr}$ is \emph{coradically graded}, i.e, $C_{\gr}$ is an $\mathbb{N}$-graded cocommutative coalgebra such that the induced filtration is the coradical filtration.

\item[2.] If $C$ is an ind-Hopf algebra in $\Ver_{p}$, then $C_{\gr}$ is also a Hopf algebra and is commutative if $C$ is irreducible.

\item[3.] $\Prim(C_{\gr}) = C_{\gr}[1] = \Prim(C).$

\end{enumerate}

\end{lemma} 

We also want a slight generalization of this Lemma if $C$ is a cocommutative Hopf algebra in $\Ver_{p}^{\ind}$, rather than just a coalgebra. Let $J = \Delta^{-1}(C_{0} \otimes C_{0})$ and let $\g$ be the space of primitives in $C$. Define a filtration $F$ on $C$ inductively as 

\begin{enumerate}

\item[1.] $F_{0}(C) = J.$

\item[2.] $F_{i}(C)$ is the kernel of the composition of $\Delta: C \rightarrow C \otimes C$ with the natural projection $C \otimes C \rightarrow C/F_{i-1}(C) \otimes C/F_{0}(C)$.

\end{enumerate}

\begin{remark} This is a coradical filtration on $C$ relative to $J$. It is defined to be dual to the descending filtration on $C^{*}$ induced by powers of $C^{*}_{\not=0}$.

\end{remark}

This filtration has the following analogous properties.

\begin{proposition} \label{relativecoradical} Let $C$, $J$ and $\g$ be as above and let $F$ be the relative coradical filtration of $C$ with respect to $J$. Then 

\begin{enumerate}

\item[1.] $F_{i}(C)$ are subcoalgebras of $C$.

\item[2.] $m(F_{i}(C) \otimes F_{j}(C)) \subseteq F_{i+j}(C).$

\item[3.] $\gr_{F}(C)$ is a graded cocommutative Hopf algebra with $\gr_{F}(C)[0] = J$ and $\Prim(\gr_{F}(C)) = \g$. 

\end{enumerate}

\end{proposition}

\begin{proof} The proof of statement 1 and 2 follow in exactly the same way as in the case of the coradical filtration. We only prove statement 3 here. Statement 1 and 2 immediately imply that the associated graded is a cocommutative Hopf algebra and the degree $0$ piece is $J$ by definition. Choose some lift $\g' \subseteq C$ of $\Prim(\gr_{F}(C))$. Then $(\Delta - \id \otimes \iota - \iota \otimes \id)(\g') \subseteq J \otimes J.$
Hence, $\g'_{0} \susbeteq J$, as $\Delta(\g'_{0}) \subseteq C_{0} \otimes C_{0}$. Consequently, $\g'_{0} = \g_{0}$. On the other hand, the equation above also implies that $\g'_{\not=0}$ must be actually primitive in $C$, since the image of $\Delta - \id \otimes \iota - \iota \otimes \id$ lies entirely inside $C_{0} \otimes C_{0}$. Hence, $\g' = \g$. 
\end{proof}

\begin{remark} If $C$ is a cocommutative Hopf algebra in $\Ver_{p}^{\ind}$, then $C^{*}$ is a topological commutative Hopf algebra in $\Ver_{p}^{\pro}$. Hence, we can think of these cocommutative Hopf algebras as formal groups in $\Ver_{p}$. 

\end{remark}

\subsection{The Dual Coalgebra}

In this subection, we will define the dual coalgebra to an ind-algebra $A$ in $\Ver_{p}$. Most of the constructions in this section are generalizations of \cite[Chapter 6]{S}. If $H$ is a commutative ind-Hopf algebra in $\Ver_{p}$, the dual coalgebra will be a cocommutative ind-Hopf algebra $C$ that is equipped with a non-degenerate pairing with $H$.

\begin{definition} Suppose $A$ is a commutative ind-algebra in $\Ver_{p}$. Then, the dual ind-coalgebra, denoted $A^{\circ}$, is the directed union of $(A/I)^{*}$ over all cofinite ideals $I$. If $A$ is an ind-Hopf algebra with counit $\epsilon$, then $I = \ker(\epsilon)$ is the \emph{augmentation ideal} of $A$ and $(A^{\circ})^{1}$ is the directed union $\bigcup_{n=1}^{\infty} (A/I^{n})^{*}.$
\end{definition}

\begin{remark} Note that for each cofinite ideal $J$, $(A/J)^{*}$ is naturally a subobject of $A^{*}$, identified with the kernel in $A^{*}$ of the evaluation pairing of $J$ with $A^{*}$ via Proposition \ref{dualident}. So, we can take the directed union inside $A^{*}$.

\end{remark}

\begin{definition} Let $A$ be a commutative ind-algebra in $\Ver_{p}$ and $M$ an $A$-module. Then, we define $M^{\circ}$ to be the directed union of $(M/IM)^{*}$ over all cofinite ideals $I$ of $A$, and define $(M^{\circ})^{1}$ as the directed union of $(M/\ker(\epsilon)^{n}M)^{*}$. These are $A^{\circ}$ and $(A^{\circ})^{1}$-comodules. 

\end{definition}

This dual coalgebra has the following obvious universal property.

\begin{proposition} Let $C$ be a cocommutative Hopf algebra in $\Ver_{p}$ (hence of finite length.) Then, $\mathrm{Hom}_{\mathrm{coalg}}(C, A^{\circ}) \cong \mathrm{Hom}_{\mathrm{alg}}(A, C^{*}).$

\end{proposition} 

The proof follows from the finite length of $C$ and $C^{*}$. We also have the following generalization of \cite[Lemma 6.0.1]{S}.

\begin{lemma} \label{6.0.1} Let $A, B$ be ind-algebras in $\Ver_{p}$. Let $f:A \rightarrow B$ be a homomorphism of algebras. 

\begin{enumerate}

\item[1.] $f^{*}(B^{\circ}) \subseteq A^{\circ}.$ 

\item[2.] $A^{\circ} \otimes B^{\circ} = (A \otimes B)^{\circ}$ as subobjects inside $(A\otimes B)^{*}$.

\item[3.] If $m$ is the multiplication map on $A$, then $m^{*}$ sends $A^{\circ}$ to $A^{\circ }\otimes A^{\circ}$

\end{enumerate}

\end{lemma}

\begin{proof} This proof is basically the same as that of \cite[Lemma 6.0.1]{S}. 3 follows immediately from 1 and 2 so we only prove those.

\begin{enumerate}

\item[1.] If $J$ is a cofinite ideal of $B$, $f^{-1}(J)$ is a cofinite ideal of $A$ and $f^{*}$ sends $(B/J)^{*}$ into $(A/f^{-1}(J))^{*}$. This proves statement 1. 

\vspace{0.2cm}

\item[2.] Note that $A^{\circ} \otimes B^{\circ}$ is the directed union of $(A/I)^{*} \otimes (B/J)^{*}$ over all cofinite ideals $I$ of $A$ and $J$ of $B$. But as these are finite length objects in $\Ver_{p}$, 

$$(A/I)^{*} \otimes (B/J)^{*} = ((A/I) \otimes (B/J))^{*} = [(A \otimes B)/(I\otimes B + A \otimes J)]^{*} \subseteq (A \otimes B)^{\circ}$$
as $I \otimes B + A \otimes J$ is a cofinite ideal of $A \otimes B$. Let us prove the reverse inclusion. Suppose $K$ is a cofinite ideal in $A \otimes B$. Given this, we can construct two ideals: $I \subseteq A$ and $J \subseteq B$ as the intersections of $K$ with $A \otimes \mathbf{1}$ and $\mathbf{1} \otimes B$ respectively. $K$ contains $A \otimes J + I \otimes B$. Hence, $[(A\otimes B)/K]^{*} \subseteq (A/I)^{*} \otimes (B/J)^{*} \subseteq A^{\circ} \otimes B^{\circ}.$

\end{enumerate}
\end{proof}

\begin{remark} If $A, B$ are Hopf algebras and $f$ a Hopf algebra homomorphism, then the above lemma also holds with $(A^{\circ})^{1}$ replacing $A^{\circ}$. The proof of part 1 is identical, and the proof of part 2 essentially shows that $(A/I^{n})^{*} \otimes (B/J^{m})^{*} \subseteq [(A \otimes B)/(I + J)^{m+n}]^{*}$ and $[(A \otimes B)/(I + J)^{k}]^{*} \subseteq \displaystyle\sum_{i + j = k} (A/I^{i})^{*} \otimes (B/J^{j})^{*}$
(with $I, J$ the respective augmentation ideals).

\end{remark}

\begin{proposition} \label{coalg} If $A$ is an ind-algebra in $\Ver_{p}$, then $A^{\circ}$ is a ind-coalgebra with comultiplication and counit dual to the multiplication and unit maps. If $A$ is commutative, then $A^{\circ}$ is cocommutative. If $A$ is an ind-Hopf algebra, then so are $A^{\circ}$ and $(A^{\circ})^{1}$, with the duals being cocommutative if $A$ is commutative.

\end{proposition}

\begin{proof} Using Lemma \ref{6.0.1}, the proof of this is identical to that of \cite[Proposition 6.0.2]{S} as that proof is purely diagram theoretic and duals still make sense, they just live in the pro-category.
\end{proof}

Let us examine the relationship between $H$ and $H^{\circ}$ more closely when $H$ is a finitely generated commutative Hopf algebra in $\Ver_{p}^{\ind}$. We have a canonical splitting as objects in $\Ver_{p}^{\ind}$, $H \cong \mathbf{1} \oplus I$, with $I$ being the augmentation ideal.

\begin{proposition} \label{non-degencirc} Let $H$ be a finitely generated commutative ind-Hopf algebra in $\Ver_{p}$. The evaluation pairing $H^{\circ} \otimes H \rightarrow \mathbf{1}$ is a non-degenerate Hopf pairing.

\end{proposition} 

\begin{proof} It is obvious that the left kernel is $0$ since the kernel in $H^{*}$ of the evaluation pairing is $0$. To see that there is no right kernel, we just need to prove that for every simple subobject $X$ of $H$, there is a cofinite ideal of $H$ not containing $X$. Let $I = \Ann(X)$, the largest ideal in $H$ such that $X \otimes I$ is in the kernel of the multiplication map $H \otimes  H \rightarrow H$. Let $\m$ be a maximal ideal containing $I$. Since $\m$ contains the ideal generated by all simple subobjects of $H$ not isomorphic to $\mathbf{1}$ by Lemma \ref{nilpotence}, $\m^{n}$ is cofinite for any $n$, so we just need to show that $X \subsetneq  \displaystyle\bigcap_{i=1}^{\infty} \m^{i}.$
Assume the contrary. Then, by the Krull-Intersection theorem applied to $H_{\m}$, $X$ must be killed by some $x \in H_{0} \backslash \m$. But then $x \in I \backslash \m$ which is empty. This is a contradiction. Hence, the evaluation pairing between $H^{\circ}$ and $H$ is non-degenerate. The fact that it is a Hopf pairing is obvious from how the Hopf structure on $H^{\circ}$ was defined. 
\end{proof}

For connected Hopf algebras in $\Ver_{p}$, we can say a bit more.

\begin{definition} We say that a finitely generated commutative Hopf algebra $H$ in $\Ver_{p}^{\ind}$ is \emph{connected} if $\Spec(\overline{H})$ is a connected affine group scheme of finite type over $\k$. Here $\overline{H} = H/I$, where $I$ is the ideal generated by $H_{\not=0}$, the sum of all simple subobjects of $H$ not isomorphic to $\mathbf{1}$. 

\end{definition}

\begin{proposition} \label{connected} Let $H$ be a finitely generated connected commutative ind-Hopf algebra in $\Ver_{p}$. Then the evaluation pairing $(H^{\circ})^{1} \otimes H \rightarrow \mathbf{1}$ is a non-degenerate Hopf pairing.

\end{proposition} 

\begin{proof} As in the previous proposition, we just need to prove that the right kernel is $0$. This is equivalent to saying that, for $I$ the augmentation ideal of $H$, $I^{\infty}:= \cap_{n=1}^{\infty} I^{n} = 0.$ 

We first show that $I^{\infty}$ is a Hopf ideal by showing that $\pi(\Delta(I^{\infty})) = 0,$
where $\pi: H \otimes H \rightarrow H/I^{\infty} \otimes H/I^{\infty}$
is the natural projection. To show this, we just need to show that $\Delta(I^{\infty}) \subseteq I^{n} \otimes H + H \otimes I^{n}$ for all $n$. This follows from $I^{\infty} \subseteq I^{2n}$. Hence, $I^{\infty}$ corresponds to a closed subgroup $G'$ of $G = \Spec(H)$, with $\O(G') = H/I^{\infty}.$ 

Let $\widehat{G}_{1}$ be the completion of $G$ at the identity, i.e., $\O(\widehat{G}_{1})$ is the inverse limit of $H/I^{n}$ over all $n$. Then, since the kernel of the map from $H$ to $\O(\widehat{G}_{1})$ is $I^{\infty}$, $G'$ and $G$ are isomorphic when completed at the identity. 

Homogeneity now implies that $H$ and $G$ are isomorphic when completed at every maximal ideal that contains $I^{\infty}$, i.e., at every point of $H$. Let us explain in more detail what we mean by homogeneity here. Since $H$ is a Hopf algebra, the set $G(\mathbf{1}) = \mathrm{Hom}_{\alg}(H, \mathbf{1})$ is a group. This group acts on $\O(G)$ by algebra automorphisms. If $f \in \mathrm{Hom}_{\alg}(H, \mathbf{1})$, then the automorphism corresponding to $f$ is $(f \otimes 1) \circ \Delta_{H}$. The set of maximal ideals in $\O(G)$ is a torsor for this group. Homogeneity is the transitive action of this group on the set of maximal ideals. Hence, if we use this action in $G'$, we get the local isomorphism between $G'$ and $G$ at every closed point in $G'$, i.e., isomorphisms after completion at every maximal ideal containing $I^{\infty}$. 

Now, as $G_{0} = \Spec(\overline{H})$ is a connected ordinary affine group scheme over $\k$, we know that $(I/J)^{\infty} = 0$, where $J$ is the ideal defining $G_{0}$. Hence, $I^{\infty} \subseteq J$. But $J$ is a nilpotent ideal by Lemma \ref{nilpotence}. Hence, $I^{\infty}$ is contained in the nilradical of $A$, and hence $G$ and $G'$ have the same closed points, the same maximal ideals. 

Thus, we have the following commutative diagram. 

$$\begin{tikzpicture}
\matrix (m) [matrix of math nodes,row sep=5em,column sep=5em,minimum width=5em]
{H & \prod_{\m} \widehat{H_{\m}}\\
\O(G') &  \prod_{\n = \m/I^{\infty}} \widehat{\O(G')_{\n}}& \\};;
\path[-stealth]
(m-1-1) edge (m-1-2)
(m-1-2) edge (m-2-2)
(m-2-1) edge (m-2-2)
(m-1-1) edge (m-2-1);
\end{tikzpicture}$$
where the product is taken over all maximal ideals $\m$ of $H$. The right map is an isomorphism, the left map is a surjection. So, to finish the proof we just need to show the top and bottom maps are both injections. But this follows from Corollary \ref{injectcomplete}.

\end{proof} 

Let us examine the structure of $H^{\circ}$ and $(H^{\circ})^{1}$ in more detail. We begin with a description of the primitives.

\begin{proposition} \label{prim} Let $G$ be an affine group scheme of finite type in $\Ver_{p}$ with function algebra $H$. Let $\g = (I/I^{2})^{*} \subseteq (H^{\circ})^{1}.$ Then, $\g = \Prim(H^{\circ})$ and $\g$ is closed under the commutator bracket on $H^{\circ}$.

\end{proposition}

\begin{proof} We will use the identification of $(I/I^2)^{*}$ as the complement of $I^{2}$ under the evaluation pairing between $H^{\circ}$ and $H$ given by Proposition \ref{dualident} and Proposition \ref{non-degencirc}. Let $m, \Delta, \epsilon, \iota$ be the Hopf algebra structure maps for $H$ and $m', \Delta', \epsilon', \iota'$ be those for $H^{\circ}$ (recall that $m' = \Delta^{*}$ and so on). 

Note that $I^{*}$ is in fact the augmentation ideal of $H^{\circ}$. Hence, if $X$ is a subobject of $\Prim(H^{\circ})$, then by the counit axioms, it is immediate that $X \subseteq I^{*}$. Now, to check that the evaluation pairing kills $X \otimes I^{2}$, it suffices to check $\Delta'(X)$ kills $I \otimes I$ under the evaluation pairing, as this pairing is a Hopf pairing by Proposition \ref{non-degencirc}. But $\Delta'$ is the same as $\iota' \otimes \id + \id \otimes \iota'$ on $X$ and the image of $\iota'$ is the kernel of $I$ under the evaluation pairing. Hence, $\Prim(H^{\circ}) \subseteq \g$.

On the other hand, $\g \subseteq I^{*}$, the augmentation ideal in $H^{\circ}$, and by the counit axiom,

$$(\Delta' - \iota' \otimes \id - \id \otimes \iota') (I^{*}) \subseteq I^{*} \otimes I^{*}.$$
Since $\g$ is the kernel of $I^{2}$, again by the fact that evaluation is a Hopf pairing,

$$(\Delta' - \iota' \otimes \id - \id \otimes \iota')(\g) \cap (I^{*} \otimes I^{*}) = 0.$$
Hence, $\g \subseteq \Prim(H^{\circ})$, as desired. The rest follows from the standard fact that primitives are closed under bracket since $(\Delta' - \iota' \otimes \id - \id \otimes \iota') $ is a Lie homomorphism from $H^{\circ}$ to itself.

\end{proof} 

In fact, we can say more.

\begin{proposition} $(H^{\circ})^{1} = H^{\circ}_{1}$, the irreducible component of $H^{\circ}$ containing the unit.

\end{proposition} 

\begin{proof} Let $I$ be the augmentation ideal in $H$. $(H^{\circ})^{1}$ is the restricted dual to the local algebra $\widehat{H_{I}}$ and is hence an irreducible coalgebra. Hence, it is contained in $H^{\circ}_{1}$. The reverse follows from Corollary \ref{coideal-primitives} and Proposition \ref{prim}.
\end{proof}

To finish the description of $H^{\circ}$ as a coalgebra, we need to describe its grouplike elements.

\begin{lemma} The grouplike elements of $H^{\circ}$ correspond to $\Spec(H)(\k) = \mathrm{Hom}_{\mathrm{alg}}(\mathbf{1}, \O(H))$. 

\end{lemma}

\begin{proof} Grouplike elements correspond to coalgebra homomoprhisms $\mathbf{1} \rightarrow H^{\circ}$, which are the same as algebra homomorphisms $H \rightarrow \mathbf{1}$. Hence, the grouplike elements of $H^{\circ}$ correspond to $\Spec(H)(\k).$ 
\end{proof}

\begin{corollary} \label{coalgebra-semidirect} Let $H$ be a finitely generated commutative Hopf algebra in $\Ver_{p}^{\ind}$. Then $H^{\circ} = \k \Spec(H)(\k) \otimes (H^{\circ})^{1}$
as a coalgebra, where $\k \Spec(H)(\k)$ is the free $\k$-vector space on the set $\Spec(H)(\k)$ and the coalgebra structure on $\k \Spec(H)(\k)$ is defined by making $\Spec(H)(\k)$ grouplike.

\end{corollary}

\section{\Large{\textbf{Harish-Chandra pairs and dual Harish-Chandra pairs in $\Ver_{p}$}}}

In this section, we will define Harish-Chandra pairs in $\Ver_{p}$ via dual Harish-Chandra pairs and show that we have a functor from the category of affine group schemes of finite type in $\Ver_{p}$ to the category of Harish-Chandra pairs in $\Ver_{p}.$ To do so, we first need to carefully define Lie algebras in $\Ver_{p}$. 

\subsection{Lie algebras in symmetric tensor categories in characteristic $p$}

The definition of a Lie algebra in $\Ver_{p}$ is not as elementary as it sounds. This entire section is largely a transcription of \cite[Section 4]{Et1}, kept here for the convenience of the reader.

\begin{definition} Let $\C$ be a symmetric tensor category in characteristic $p$. An \emph{operadic Lie (ind-)algebra} is an (ind-)object $L \in \C$ equipped with the a map $[-,-]: \wedge^{2}L \rightarrow L$ that satisfies the Jacobi identity

$$[-,-] \circ ([-,-] \otimes \id_{L}) \circ (\id_{L^{\otimes 3}} + c_{123} + c_{123}^{2})(L^{\otimes 3}) = 0$$
where $c_{123}$ is the $3$-cycle $(123) \in S_{3}$ acting on $L^{\otimes 3}$ by permuting the tensor factors. 

\end{definition}

\begin{remark} Note that even for $\C = \Vec$, an operadic Lie algebra is not a Lie algebra in characteristic $2$, as the relation $[x, x] = 0$ is missing. Similarly, if the characteristic is $3$, then for $\C = \sVec$, we are missing the relation $[x, [x, x]] = 0$ for odd elements $x$. It is no surprise that we need some additional relations to define Lie algebras in symmetric tensor categories in general.

\end{remark}

\begin{example} Associative algebras are operadic Lie algebras under the commutator bracket.

\end{example}

Here is an alternative way to present this definition. Recall the notion of the Lie operad (see \cite{LV}) ${\bf Lie} :=\displaystyle\bigoplus_{n \ge 1} {\bf Lie}_{n}$
generated over $\Z$ by a single antisymmetric element $b \in {\bf Lie}_{2}$ with Jacobi identity as the defining relation. An operadic (ind-)Lie algebra in $\C$ is an (ind-)object equipped with the structure of an algebra over ${\bf Lie}$.

Note that ${\bf Lie}_{n}$ is equipped with the natural action of the symmetric group $S_{n}$. Additionally, the braiding in a symmetric tensor category induces an $S_{n}$ action on $V^{\otimes n}$ for any object $V$. Hence, we can define a free operadic Lie algebra as follows.

\begin{definition} Let $V \in \C$. Define the \emph{free operadic Lie algebra} $\FOLie(V)$ as

$$\FOLie(V) = \bigoplus_{n \ge 1} \FOLie_{n}(V) = \bigoplus_{n \ge 1} (V^{\otimes n} \otimes {\bf Lie_{n}})_{S_{n}}$$
where the subscript indicated coinvariants.

\end{definition}

This has an obvious bracket induced by $\mathbf{Lie}$ which makes it an operadic Lie algebra. Moreover, it is generated as an operadic Lie algebra in degree $1$ and has a universal property that immediately follows from the definition.

\begin{proposition} The space of Lie (i.e. bracket-preserving) homomorphisms from $\FOLie(V)$ to any operadic Lie algebra $L$ is in natural bijection with $\mathrm{Hom}_{\C^{\ind}}(V, L)$.

\end{proposition}

In particular, we have a natural Lie algebra map $\phi^{V}: \FOLie(V) \rightarrow TV$
induced by the inclusion of $V$ into its tensor algebra. Let $\phi_{n}^{V}$ be the restriction of this map to $\FOLie_{n}(V) \rightarrow V^{\otimes n}$ and define $E_{n}(V):= \Ker(\phi_{n}^{V}) \text{ and } E(V) = \displaystyle\bigoplus_{n\ge 1} E_{n}(V).$

\begin{definition} A \emph{Lie algebra} $L$ in $\C$ is an operadic Lie algebra such that the natural map $\beta^{L}: \FOLie(L) \rightarrow L$
induced by the identity on $L$ is $0$ on $E(L)$. 

\end{definition}

This definition seems somewhat involved, since $E_{n}(V)$ can be fairly tricky to compute for large values of $n$. However, for our purposes, we have the following very nice fact.

\begin{proposition} \label{Lie-assoc} Any associative algebra $A$ or its subobject closed under bracket is a Lie algebra.

\end{proposition}

\begin{proof} Since $A$ is an associative algebra $\beta^{A}: \FOLie(A) \rightarrow A$ factors through $T(A)$ and hence, is automatically zero on $E(A)$. The case of a subobject closed under bracket follows immediately.
\end{proof}

The only operadic Lie algebras we will consider in this paper are those that arise as Lie subalgebras of associative algebras and are hence automatically Lie algebras. 

\subsection{Lie algebra of an affine group scheme in $\Ver_{p}$ and the underlying ordinary affine group scheme}
 
Let $G$ be an affine group scheme in $\Ver_{p}$ and let $H= \O(G)$ be its ind-algebra of functions. Then, $H$ is a commutative ind-Hopf algebra in $\Ver_{p}$. Let $I$ be its augmentation ideal. Note that $H$ has a canonical decomposition as $\mathbf{1} \oplus I$ via the unit and counit maps. 

\begin{definition} The Lie algebra of $G$, denoted $\Lie(G)$ or $\g$, is $(I/I^{2})^{*} \subseteq H^{\circ}$ in $\Ver_{p}^{\ind}$.

\end{definition} 

Let us elaborate on some properties of $\Lie(G)$.

\begin{proposition} If $G$ is of finite type, then $\g$ is an object in $\Ver_{p}$ of finite length.

\end{proposition}

\begin{proof} We show that for any maximal ideal $\m$ in a finitely generated commutative algebra $A$ in $\Ver_{p}^{\ind}$, $\m/\m^{2} \in \Ver_{p}$. Since $A$ is finitely generated, we have a surjection $f: S(X) \rightarrow A.$ $f^{-1}(\m)$ is a maximal ideal in $S(X)$ and $\m/\m^{2}$ is bounded in length by $f^{-1}(\m)/(f^{-1}(\m))^{2}$. Hence, we can reduce to the case where $A$ is a symmetric algebra. By Lemma \ref{nilpotence}, $S(X) = \k[x_{1}, \ldots x_{n}] \otimes Y,$ where $Y$ is a commutative algebra in $\Ver_{p}$ (hence of finite length). Then, $\m = \m_{1} \otimes Y + \k[x_{1}, \ldots, x_{n}] \otimes \m_{2}$, with $\m_{1}, \m_{2}$ maximal ideals in the respective tensor factors. This implies that $\m/\m^{2} = \m_{1}/\m_{1}^{2} \oplus \m_{2}/\m_{2}^{2},$
and from here the result follows from classical commutative algebra and $Y$ and hence $\m_{2}$ being finite length.
\end{proof}

To justify the terminology of a Lie algebra, Proposition \ref{prim} tells us that $\g$ is the space of primitives inside $A^{\circ}$. Hence, $\g$ is a subobject of an associative algebra closed under commutator.

\begin{proposition} \label{LieisLie} If $G$ is a an affine group scheme in $\Ver_{p}$ of finite type, then $\g$ is a Lie algebra in $\Ver_{p}$.

\end{proposition}

As $\g$ is the space of primitives inside $H^{\circ}$, it also acquires the structure of a left $H^{\circ}$-module.

\begin{definition} The \emph{left adjoint action} of $H^{\circ}$ on itself is given by the action map $\ad: H^{\circ} \otimes H^{\circ} \rightarrow H^{\circ}$
where $\ad$ is the composite map

$$\begin{tikzpicture}
\matrix (m) [matrix of math nodes,row sep=3em,column sep=3em,minimum width=3em]
{ H^{\circ} \otimes H^{\circ} & H^{\circ} \otimes H^{\circ} \otimes H^{\circ} & H^{\circ} \otimes H^{\circ} \otimes H^{\circ} & H^{\circ} \otimes H^{\circ} \otimes H^{\circ}  & H^{\circ} \\};;
\path[-stealth]
(m-1-1) edge node[auto] {$\Delta_{1}$} (m-1-2)
(m-1-2) edge node[auto] {$c_{2,3}$} (m-1-3)
(m-1-3) edge node[auto] {$S_{3}$} (m-1-4)
(m-1-4) edge node[auto] {$m$} (m-1-5);
\end{tikzpicture}$$
Here, $\Delta_{1}$ is comultiplication in the first component, $c_{23}$ is the swap map in the second and third component, $S_{3}$ is the antipode on the third component and $m$ is multiplication.

\end{definition}

Since $\g$ is the space of primitives inside $H^{\circ}$, which is a cocommutative Hopf algebra in $\Ver_{p}^{\ind}$, we have the following proposition.

\begin{proposition} \label{adjoint} $\g$ is a submodule of $H^{\circ}$ under the left adjoint action. \end{proposition}

\begin{proof} The standard proof is element free and works perfectly in this setting too. It uses the above diagram and the antipode axiom only.
\end{proof} 

In addition to the Lie algebra of an affine group scheme, we also have an underlying ordinary affine group scheme.

\begin{definition} Let $G$ be an affine group scheme of finite type in $\Ver_{p}$ with algebra of functions $H$. Let $J$ be the ideal in $H$ generated by all simple subobjects not isomorphic to $\mathbf{1}$. Then, the \emph{underlying ordinary affine subgroup scheme}, denoted $G_{0}$ is $\Spec(\overline{H})$, with $\overline{H} = H/J.$

\end{definition} 

Semisimplicity of $\Ver_{p}$ immediately implies that $J$ is a Hopf ideal in $H$ and hence $\overline{H}$ is a finitely generated commutative Hopf algebra over $\k$. We end this subsection with the following compatibility between $G_{0}$ and $\g$.

\begin{proposition} \label{Lie-Vec} Let $G$ be an affine group scheme of finite type in $\Ver_{p}$ with Lie algebra $\g = \g_{0} \oplus \g_{\not=0}$, where $\g_{0}$ is the isotypic component of $\g$ coming from $\mathbf{1}$ and $\g_{\not=0}$ is the direct sum of all other isotypic components. Then, $\g_{0}$ is a Lie subalgebra of $\g$ and is isomorphic to $\Lie(G_{0})$. 

\end{proposition} 

\begin{proof} The fact that $\g_{0}$ is a Lie subalgebra of $\g$ is immediate from the fact that $\mathbf{1} \otimes \mathbf{1} \cong \mathbf{1}$. For the second half of the proposition, we will prove the dual statements. Let $H$ be the algebra of functions of $G$ and let $I$ be the augmentation ideal. Let $J$ be the ideal which we quotient by to get $\overline{H}$. Since $J$ is nilpotent by Lemma \ref{nilpotence} and $I$ is a maximal ideal, $J \subseteq I$. Write $I = I' \oplus J$, picking some arbitrary lift $I' \cong I/J$ in $\Ver_{p}^{\ind}$. It is thus immediate that $(I/I^{2}) \text{ mod } J$ is the same as $I'/(I')^{2} \text{ mod } J$.

Now, $\g_{\not=0} \subseteq J/I^{2} \subseteq I/I^{2}$. Hence, $\g_{0} \subseteq I'/I^{2} = I'/(I')^{2} \text{ mod } J$. The reverse inclusion is obvious as $I'$ has only $\mathbf{1}$ as a simple subobject.
\end{proof} 

Motivated by this proposition, we have the following definition.

\begin{definition} Let $\g$ be an ind-Lie algebra in $\Ver_{p}$. The \emph{underlying ordinary Lie subalgebra}, denoted $\g_{0}$, is the isotypic component of $\g$ coming from $\mathbf{1}$.

\end{definition}

\subsection{PBW theorem for Lie algebras in $\Ver_{p}$}

\begin{definition} Let $\g$ be an operadic Lie algebra in $\Ver_{p}$. The universal enveloping algebra $U(\g)$ is the quotient of the tensor algebra $T(\g)$ by the ideal generated by the image of $a: \g \otimes \g \rightarrow \g \oplus \g^{\otimes 2} \subseteq T(\g)$
where $a$ is the difference between the commutator in $T(\g)$ and the Lie bracket on $\g$.

\end{definition}

This universal enveloping algebra satisfies the standard universal property.

\begin{proposition} The space of unital algebra homomorphisms from $U(\g)$ to any associative, unital ind-algebra $A \in \Ver_{p}$ is naturally isomorphic to the space of Lie algebra homomorphisms from $\g$ to $A$.

\end{proposition}

Note that $U(\g)$ is a filtered quotient of $T(\g)$. Taking associated graded objects gives us an algebra homomorphism $S(\g) \rightarrow \gr U(\g)$
that is always surjective.

\begin{definition} We say that an operadic Lie algebra $\g$ in $\Ver_{p}$ \emph{satisfies PBW} if this map is an isomorphism.

\end{definition}

The question of which operadic Lie algebras satisfy the PBW theorem is fairly involved and is studied extensively in \cite{Et1}. One useful result from that article is the following (\cite[Theorem 6.6]{Et1}).

\begin{proposition} Let $\g$ be an operadic Lie algebra in $\Ver_{p}$. Then the following are equivalent:

\begin{enumerate}

\item[1.] $\g$ is a Lie algebra.

\item[2.] $\g$ satisfies PBW.

\end{enumerate}

\end{proposition}

By Proposition \ref{LieisLie}, we have the following consequence.

\begin{corollary} \label{PBWgroup} Let $G$ be an affine group scheme of finite type in $\Ver_{p}$. Then, $\Lie(G)$ satisfies PBW.

\end{corollary}

\subsection{Dual Harish-Chandra pairs and Harish-Chandra pairs}

In this section we finally give the formal definition of a Harish-Chandra pair. Informally, the data of a Harish-Chandra pair is an ordinary affine group scheme of finite type $G_{0}$, a Lie algebra $\g$ in $\Ver_{p}$ and compatibility between $\Lie(G_{0})$ and $\g_{0}$. More formally, we first need the notion of a dual Harish-Chandra pair and then we define Harish-Chandra pairs by dualizing.

\begin{definition} A \emph{dual Harish-Chandra pair} in $\Ver_{p}$ is a pair $(J, \g)$, where $J$ is a cocommutative Hopf algebra in $\Vec$ and $\g$ is a Lie algebra in $\Ver_{p}$ that is also a left $J$-module, equipped with an isomorphism $i: \Prim(J) \rightarrow \g_{0}$ of ordinary Lie algebras such that:

\begin{enumerate}
\item[1.] The bracket on $\g$ is a $J$-module map.
\item[2.] The map $i$ is an isomorphism of $J$-modules, with $\Prim(J)$ given the left adjoint action of $J$.
\item[3.] The two actions of $\Prim(J)$ on $\g$ via the $J$-module action and the adjoint action of $\g_{0}$ coincide.

\end{enumerate}

\end{definition}

\begin{remark} While the morphism $i$ allows for a more precise definition of a dual Harish-Chandra pair, it is largely irrelevant in applications, and we can simply think of a dual Harish-Chandra pair in $\Ver_{p}$ as a pair $(J, \g)$ of a cocommutative ind-Hopf algebra $J$ over $\k$ and a Lie algebra $\g \in \Ver_{p}$ with $\Prim(J) = \g_{0}$, such that the adjoint action of $J$ on $\g_{0} = \Prim(J)$ extends to an action of $J$ on $\g$ that restricts to the adjoint action of $\Prim(J) = \g_{0}$ on $\g$.

\end{remark}

Harish-Chandra pairs are defined via duality. 

\begin{definition} A \emph{Harish-Chandra pair} in $\Ver_{p}$ is a pair $(H, W)$ of a finitely generated commutative Hopf algebra $H$ in $\Vec$ and a right $H$-comodule $W$ in $\Ver_{p}$ such that $(H^{\circ}, W^{*})$ is equipped with the structure of a dual Harish-Chandra pair.

\end{definition} 

\begin{remark} Note that if $W$ is a right comodule for $H$, then $W^{*}$ is naturally a left module for $H^{\circ}$ via the following procedure: take coevaluation for $W$ on the right to get a map

$$H^{\circ } \otimes W^{*} \rightarrow H^{\circ} \otimes W^{*} \rightarrow W \otimes W^{*}.$$
Coact on $W$ and evaluate with $W^{*}$ to get a map to $H^{\circ} \otimes H \otimes W^{*}$ and then pair $H^{\circ}$ with $H$.

\end{remark}

We want Harish-Chandra pairs to form a category. Hence, we also need the notion of a morphism of Harish-Chandra pairs. 

\begin{definition} A morphism between dual Harish-Chandra pairs $(J, \g), (J', \g')$ in $\Ver_{p}$ is a pair $(f, \rho)$ where $f: J \rightarrow J'$ is a homomorphism of Hopf algebras over $\k$, $\rho : \g \rightarrow \g'$ is a morphism of Lie algebras in $\Ver_{p}$ that is a morphism of left $J$-modules (with the left $J$-action on $\g'$ coming from $f$), such that $\rho|_{\g_{0}} \circ i = f|_{\Prim(J)}.$

\end{definition}

\begin{remark} Note that this is not the same as the notion of Harish-Chandra pairs that already exists for ordinary algebraic groups. The terminology used here comes from \cite{M1}. 

\end{remark}

\begin{definition} A morphism between Harish-Chandra pairs $(H, W), (H', W')$ is a pair $(f, \rho)$ with $f$ a Hopf algebra homomorphism from $A$ to $A'$ over $\k$ and $\rho$ a comodule map from $W$ to $W'$ such that $(f^{\circ}, \rho^{*})$ have the structure of a morphism of dual Harish-Chandra pairs.

\end{definition}

\begin{remark} It is clear from this definition that we get categories of Harish-Chandra pairs and dual Harish-Chandra pairs in $\Ver_{p}$ and that the category of Harish-Chandra pairs is equipped with a functor $\mathbf{D}$ to the category of dual Harish-Chandra pairs. Morally speaking, dual Harish-Chandra pairs are the same as cocommutative Hopf algebras, which are essentially formal group schemes, while Harish-Chandra pairs are the same as affine group schemes. Hence, the functor $\mathbf{D}$ is roughly the same as taking the distribution algebra dual to functions on the formal neighborhood at the identity.

\end{remark}

The constructions of the last section allow us to associate a Harish-Chandra pair to any affine group scheme of finite type in $\Ver_{p}$.

\begin{theorem} \label{forwardfunctor} If $G$ is an affine group scheme of finite type in $\Ver_{p}$, then there is a natural structure of a Harish-Chandra pair on $(\O(G_{0}), \g^{*})$. This defines a functor $\mathbf{HC}$ from the category of affine group schemes of finite type in $\Ver_{p}$ to Harish-Chandra pairs in $\Ver_{p}$.

\end{theorem}

\begin{proof} Let $A = \O(G)$ and $H = \overline{H} = \O(G_{0})$. It is clear from the previous section that $J:= \overline{H}^{\circ} $ is a cocommutative Hopf algebra over $\k$ and that $\g$ is a Lie algebra in $\Ver_{p}$ that acquires a left action of $J$ via the adjoint action. Note that the adjoint action of $\g$ on itself induced from $J$ is the same as the adjoint action coming from the Lie algebra structure on $\g$, since the antipode on primitive elements is just $-\id$. This proves everything else we need, since we have already checked that $\Prim((A/I)^{\circ}) = \g_{0}$
in Proposition \ref{Lie-Vec} and Proposition \ref{prim}.
\end{proof} 

We can now restate Theorem \ref{Harish-Chandra} more precisely as stating that $\mathbf{HC}$ is an equivalence of categories. To prove this we will use a related functor in the dual cocommutative setting.

\begin{theorem} \label{coforwardfunctor} Let $C$ be a cocommutative Hopf algebra in $\Ver_{p}^{\ind}$. Let $C_{0}$ be its $\mathbf{1}$-isotypic component. Then, $(\Delta^{-1}(C_{0} \otimes C_{0}), \Prim(C))$ has the natural structure of a dual Harish-Chandra pair in $\Ver_{p}$ and we get a functor $\mathbf{DHC}$ from the category of cocommutative Hopf algebras in $\Ver_{p}^{\ind}$ to the category of dual Harish-Chandra pairs in $\Ver_{p}$.

\end{theorem} 

\begin{proof} 

This is part of Theorem \ref{forwardfunctor}.
\end{proof}

We have the compatibility property that is immediate from the definitions and Proposition \ref{dualident}. 

\begin{proposition} \label{HC-DHC} Let $G$ be an affine group scheme of finite type in $\Ver_{p}$. Then $\mathbf{D} \circ \mathbf{HC}(G) = \mathbf{DHC}(\O(G)^{\circ}).$

\end{proposition}

\subsection{Tensor algebras and coalgebras}

\begin{definition} Let $X$ be an object in $\Ver_{p}$. The tensor algebra $T(X)$ is the Hopf algebra which has the same algebra structure as the ordinary tensor algebra of $X$ and the comultiplication is the unique one in which $X$ is primitive. Explicitly, if $\iota$ is the unit map $\mathbf{1} \rightarrow T(X)$ and $\Delta$ the comultiplication map on $T(X)$, then $\Delta : X \rightarrow T(X) \otimes T(X)$
is the map $\id \otimes \iota + \iota \otimes \id$ (identifying $X$ with $X \otimes \mathbf{1}$ and $\mathbf{1} \otimes X$).

\end{definition}

\vspace{0.2cm}

\begin{definition} Let $X$ be an object in $\Ver_{p}$. The tensor coalgebra $T_{c}(X)$ is the Hopf algebra that is the graded dual to $T(X^{*}).$ Explicitly, if $\Delta$ is the comultiplication and $m$ the multiplication,

$\Delta: X^{\otimes n} \rightarrow \bigoplus_{i+j = n} X^{\otimes i} \otimes X^{\otimes j}$
is the sum of all the natural identifications $X^{\otimes n} \cong X^{\otimes i} \otimes X^{\otimes j}$ and $m : X^{\otimes i} \otimes X^{\otimes n-i} \rightarrow X^{n}$
is the shuffle product $\displaystyle{\sum_{\tau^{-1} \in S_{n, i}}} \tau,$ where $S_{n, i} = \{\sigma \in S_{n} : \sigma(1) < \cdots < \sigma(i), \sigma(i+1) < \cdots < \sigma(n)\}$
is the set of $i$-shuffles in $S_{n}$. Here the action of a permutation comes from the braiding $c$, since a symmetric braiding induces an action of $S_{n}$ on $X^{\otimes n}$.

\end{definition}

\begin{remark}

Note that $T(X)$ is a cocommutative Hopf algebra in $\Ver_{p}^{\ind}$ and $T_{c}(X^{*})$ is a commutative Hopf algebra in $\Ver_{p}^{\ind}$ and there is a nondegenerate $\mathbb{N}$-graded Hopf pairing between the two.

\end{remark}

\begin{definition} If $C$ is a cocommutative Hopf algebra in $\Ver_{p}^{\ind}$, and $X$ is a left $C$-module in $\Ver_{p}$, then we turn $T(X)$ into a left $C$-module Hopf algebra via the diagonal action

$$C \otimes X^{\otimes n} \rightarrow (C\otimes \cdots \otimes \mathbf{1} \oplus \cdots \oplus \mathbf{1} \otimes \cdots \otimes C) \otimes (X^{\otimes n}) \rightarrow X^{\otimes n}$$
where the first map is just $n$-fold comultiplication and the second map is the action in each tensor component, using the braiding to move tensor factors around.

Similarly, if $A$ is a commutative Hopf algebra and $X$ is a right $A$-comodule, we turn $T_{c}(X)$ into a right $A$-comodule Hopf coalgebra via the map $\rho : X^{\otimes n} \rightarrow X^{\otimes n} \otimes A$
by coacting in each component, moving all the components of $A$ next to each other using the braiding and then multiplying in $A$.

\end{definition}

\begin{remark} This construction turns $T(X)$ into a Hopf algebra object in the category of left $C$-modules and $T_{c}(X)$ into a Hopf algebra object in the category of right $A$-comodules.

\end{remark}

We end this section with a construction of smash products in this specialized setting of tensor algebras and coalgebras.

\begin{definition} Let $C$ be a cocommutative Hopf algebra in $\Ver_{p}^{\ind}$ and $X$ a left $C$-module. There is a Hopf algebra structure on $T(X) \otimes C$ as follows.

\begin{enumerate}

\item[1.] The unit map is just $\iota_{C} \otimes \iota_{T(X)}.$ 

\item[2.] The counit map is also just $\epsilon_{C} \otimes \epsilon_{T(X)}.$

\item[3.] The antipode is also simply $S_{C} \otimes S_{T(X)}$. 

\item[4.] Comultiplication is $\Delta_{C} \otimes \Delta_{T(X)}$ followed by $c$ in the middle. 

\item[5.] Multiplication $T(X) \otimes C \otimes T(X) \otimes C \rightarrow T(X) \otimes C$
is obtained by first comultiplying in the left $C$ tensor factor to obtain $C \otimes C$ in the middle, then using the braiding to permute the second of these $C$ factors past the $T(X)$ and acting on $T(X)$ by the leftmost $C$ factor to land in $T(X) \otimes T(X) \otimes C \otimes C$, and then multiplying in each factor. In Sweedler notation using elements and suppressing the braiding, this can be written as $(x, c)(x',c') = (xc_{(1)}(x'), c_{(2)}c').$

\end{enumerate}

Similarly, if $A$ is a commutative ind-Hopf algebra in $\Ver_{p}$ and $X$ is a right comodule for $A$ in $\Ver_{p}$, then we can put a Hopf algebra structure on $A \otimes T_{c}(X)$. As in the above situation, the unit, counit, antipode and now multiplication are just the tensor products (using the braiding as necessary to move factors around). Comultiplication is twisted in a dual manner to the way multiplication is twisted is above:

$$\Delta: A \otimes T_{c}(X) \rightarrow (A \otimes T_{c}(X)) \otimes (A \otimes T_{c}(X))$$
is obtained by first comultiplying in $A$ and $T_{c}(X)$ to land in $A \otimes A \otimes T_{c}(X) \otimes T_{c}(X)$, then permuting the right $A$ past the $T_{c}(X)$ and coacting in the left $T_{c}(X)$ to get $A \otimes T_{c}(X) \otimes A \otimes A \otimes T_{c}(X)$ and finally multiplying in $A\otimes A.$ 

We denote these algebras as $T(X) \rtimes C$ and $A \ltimes T_{c}(X)$ respectively and call them the \emph{smash product}.

\end{definition}

An important property of the construction is the following, readily checked from the definition.

\begin{proposition} \label{smash} If $C$ is a cocommutative ind-Hopf algebra in $\Ver_{p}$ and $X$ is a left $C$-module, then 

\begin{enumerate}

\item[1.] $C \cong \mathbf{1} \otimes C$ and $T(X) \cong T(X) \otimes \mathbf{1}$ are cocommutative Hopf subalgebras of $T(X) \rtimes C.$

\item[2.] In $T(X) \rtimes C$, the left adjoint action of $C$ preserves $T(X)$ and coincides with the original action of $C$ on $T(X).$ 

\item[3.] $T(X)_{\ge 1}$, the positive degree tensors, form an ideal in $T(X) \rtimes C$ and $C$ is the quotient by this ideal.

\item[4.] Powers of $X$ generate $T(X) \rtimes C$ as a right $C$-module via right multiplication. In particular, the smash product is generated as an algebra in $\Ver_{p}^{\ind}$ by $X$ and $C$, subject to the relation that equates the left adjoint action of $C$ on $X$ with the original one. 

\item[5.] Let $I$ be the ideal in $T(X)$ generated by the image of $c_{X, X} - \id_{X \otimes X} \subseteq X^{\otimes 2}$. Then, $I \rtimes J$ is a Hopf ideal in $T(X) \rtimes C$ and hence we get a smash product Hopf algebra structure on $S(X) \rtimes C$ as well.

\end{enumerate}

Dual statements hold for commutative ind-Hopf algebras $A$ and right comodules $X$.

\end{proposition}

We end this section with another important example of a smash product. The proof of the following theorem follows from Corollary \ref{coalgebra-semidirect}.

\begin{theorem} \label{Kostant} If $G$ is an affine group scheme of finite type in $\Ver_{p}$ with the ind-Hopf algebra of functions $H$, then, as a Hopf algebra in $\Ver_{p}^{\ind}$, $H^{\circ} \cong (H^{\circ})^{1} \rtimes \k G(\k),$
where $G(\k) = \mathrm{Hom}_{alg}(H , \k)$ acts on $(H^{\circ})^{1}$ via the dual of the conjugation action on $H$ (which preserves the augmentation ideal). Here, $\k G(\k) $ is the group algebra on $G(\k)$ (with its standard Hopf algebra structure) and the conjugation action of $g: H \rightarrow \mathbf{1} \in G(\k)$ on $H$ is described as $H \rightarrow H \otimes H \otimes H \rightarrow H,$
where the first map is $\Delta^{2}$ and the second is $g$ in the first component and $g^{-1}$ in the last.

\end{theorem}

\section{\Large{\textbf{Construction of an inverse to the functor $\mathbf{DHC}$ via PBW theorems}}}

In this section, we construct a functor from the category of dual Harish-Chandra pairs in $\Ver_{p}$ to the category of cocommutative ind-Hopf algebras in $\Ver_{p}$ that is inverse to $\mathbf{DHC}$. We use the notation from \cite{M1}. Throughout this section, we will use $(J, \g)$ to denote a dual Harish-Chandra pair in $\Ver_{p}$ and $C$ to denote a cocommutative Hopf algebra in $\Ver_{p}^{\ind}$. 

\begin{definition}  Define the Hopf algebra $\H(J, \g)$ as the Hopf smash product $T(\g_{\not=0}) \rtimes J.$

\end{definition} 

Note that $\g_{0} = \Prim(J)$ and that the action of $\g_{0}$ on $\g_{\not=0} \subseteq T(\g)$ is the adjoint action,  by definition of a dual Harish-Chandra pair and the Hopf smashed product. Hence, we can identify $\g$ as a subobject of $\H(J, \g)$. 

\begin{definition} Define $I(J, \g)$ as the ideal in $\H(J, \g)$ generated by the image of the difference between the commutator map $\g_{\not=0} \otimes \g_{\not=0} \rightarrow T(\g_{\not=0}) \subseteq \H(J, \g)$
and the Lie bracket map  $\g_{\not=0} \otimes \g_{\not=0} \rightarrow \g \subseteq \H(J, \g).$
Define $H(J, \g)$ as the quotient of $\H(J, \g)$ by $I(J, \g)$.

\end{definition}

\subsection{PBW filtrations for dual Harish-Chandra pairs}

The functor that sends $(J, \g)$ to $H(J, \g)$ will be the inverse to $\mathbf{DHC}$ that we desire. To show that this is the case, we need some additional constructions. We begin by defining another cocommutative Hopf algebra associated to $(J, \g)$. 

\begin{definition} Define $U(J, \g) := U(\g) \otimes_{U(\g_{0})} J$
as an object in $\Ver_{p}^{\ind}$ that is the quotient of $U(\g) \otimes J$ by the image of  $R_{0} - L_{0}: U(\g) \otimes U(\g_{0}) \otimes J \rightarrow U(\g) \otimes J$
where $R_{0}$ is right multiplication by $U(\g_{0})$ in $U(\g)$ and $L_{0}$ is left multiplication by $U(\g_{0})$ in $J$.

\end{definition} 

We can view $U(\g) \otimes J$ as a quotient of the Hopf smash product $T(\g) \rtimes J$ and denote this Hopf algebra as $U(\g) \rtimes J$. It is clear that the image of $R_{0} - L_{0}$ is a coideal of $U(\g) \otimes J$ with this Hopf algebra structure, as both coalgebras are cocommutative. This image is also preserved by the antipode. 

\begin{proposition} \label{semidirect} The image of $R_{0} - L_{0}$ is an ideal in $U(\g) \rtimes J$ and hence we get a Hopf algebra structure on $U(\g) \otimes_{U(\g_{0})} J$, which we denote by $U(J, \g)$ as well.

\end{proposition}

\begin{proof} It is instructive to first give the proof of the proposition if $\g$ is a Lie superalgebra rather than one in $\Ver_{p}$. In this proof we can use elements for clarity, so we use Sweedler notation for comultiplication here, namely $\Delta(x) = x_1 \otimes x_2$
with an implicit summation, and $\Delta^{2}(x) = x_1 \otimes x_2 \otimes x_3.$ We also use $x(y)$ to denote the action of $x$ on $y$ if $x \in J, y \in U(\g)$.

We can give both $U(\g) \rtimes J \otimes U(\g) \rtimes J$ and $U(\g) \rtimes J$ the structure of a $(U(\g), J)$-bimodule and multiplication in the smash product is compatible with this structure. Since $M:= \mathbf{1} \otimes J \otimes U(\g) \otimes \mathbf{1}$
generates $U(\g) \ltimes J \otimes U(\g) \ltimes J$ as a $(U(\g), J)$-bimodule, it suffices to check that for all $x \in J, y \in U(\g)$ and $z \in \g_{0}$,

$$(1 \otimes zx)(y \otimes 1) = (z \otimes x)(y \otimes 1)$$
and that

$$(1 \otimes x)(y \otimes z) = (1 \otimes x)(yz \otimes 1).$$
Now,

\begin{align*}
(1 \otimes zx)(y \otimes 1) &= z_1(x_1(y)) \otimes z_2 x_2\\
&= z_1(x_1(y))z_2 \otimes x_2\\
&= z(x_1(y)) \otimes x_2 + x_1(y)z \otimes x_2\\
&= [z, x_1(y)] \otimes x_2 + x_1(y)z \otimes x_2\\
&= z x_1(y) \otimes x_2\\
&= (z \otimes x)(y \otimes 1)
\end{align*}
since by the definition of a dual Harish-Chandra pair $\g_0$ acts on $U(\g)$ via the adjoint, i.e., commutator action. This gives the first equality. For the second, we have

\begin{align*}
(1 \otimes x)(yz \otimes 1)&= x_1(yz) \otimes x_2\\
&= x_1(y)x_2(z) \otimes x_3\\
&= x_1(y) \otimes x_2(z)x_3\\
&= x_1(y) \otimes x_2 z S(x_2) x_3\\
&= x_1(y) \otimes x_2 z\epsilon(x_2)\\
&= x_1(y) \otimes x_2 z.
\end{align*}
Here, for the second equality, we use the fact that for $U(\g)$ is a left $H$-module algebra, namely for $y, z\in U(\g)$ and $x \in J$, $x(yz) = x_1(y)x_2(z)$. For the fourth equality, we use the fact that the action of $J$ on $\g_0$ is the adjoint action by definition of a dual Harish-Chandra pair. For the fifth equality, we use the antipode axiom and for the last equality we use the counit axiom.

All of the properties we use to prove the result for $\g$ being a supervector space hold when $\g \in \Ver_{p}$ instead. Two of the facts come from the definition of a dual Harish-Chandra pair and the others follow from definitions of Hopf algebras. Additionally, we only use the action of $J$ on $\g$ and nothing special about $\g$ being a supervector space and hence having elements. Hence, this proof is easily categorified when $\g$ is a Lie algebra in $\Ver_{p}$ with the underlying ordinary Lie algebra being $\g_{0}$, i.e, when $(J, \g)$ is really a dual Harish-Chandra pair in $\Ver_{p}$. This categorical version of the proof is the same in spirit, and is less illuminating than the given proof.

\end{proof}

We view $U(J, \g)$ as a Hopf algebra in $\Ver_{p}^{\ind}$ with the above structure. Since $\g$ satisfies the PBW theorem, we have the following lemma. 

\begin{lemma} \label{PBWnormal} Filter $U(J, \g)$ by putting $\g_{\not=0}$ in degree $1$ and $J$ in degree $0$. Then, $\gr(U(J, \g)) \cong S(\g_{\not=0}) \otimes J.$
Hence, as right $J$-modules in $\Ver_{p}^{\ind}$, $U(J, \g) \cong S(\g_{\not=0}) \otimes J.$

\end{lemma}

\begin{lemma} \label{HtoUmap} The inclusion of $J, \g$ into $U(J, \g)$ induces an isomorphism of Hopf algebras $\phi: H(J, \g) \rightarrow U(J, \g)$.

\end{lemma} 

\begin{proof} $\H(J, \g)$ is generated by $J, \g$ subject to the relation that the adjoint action of $J$ on $\g$ in this cocommutative Hopf algebra is the same as the left action given in the definition of a dual Harish-Chandra pair. This relation holds in $U(J, \g)$, since it is also defined as a quotient of the smash product between $J$ and $T(\g)$. Hence, the inclusion of $J, \g$ in $U(J, \g)$ induces a homomorphism of Hopf algebras $\H(J, \g) \rightarrow U(J, \g)$. This map descends to a homomorphism $\phi: H(J, \g) \rightarrow U(J, \g)$, since the only additional relation in $H(J, \g)$ is that the commutator is the Lie bracket on $\g \subseteq H(J, \g)$, which clearly holds in $U(J, \g)$ as well. 

Similarly, $U(\g) \rtimes J$ is generated by $J$ and $\g$ subject to the commutator in $\g$ being the Lie bracket and the same relation between $J$ and $\g$ as above. Hence, we have a homomorphism of Hopf algebras $U(\g) \rtimes J \rightarrow H(J, \g)$
and this descends to a homomorphism $U(J, \g) \rightarrow H(J, \g)$. This map is clearly inverse to $\phi$, which is thus an isomorphism.
\end{proof}

We end this subsection by constructing a PBW filtration on $H(J, \g)$ that will be useful in a later subsection.

\begin{definition} Define a grading on $\H(J, \g)$ by putting $J$ in degree $0$ and $\g_{\not=0}$ in degree $1$. This descends to a filtration $F$ on $H(J, \g)$.

\end{definition}

\begin{proposition} \label{dHCPBW} The associated graded of this filtration is described as $\gr_{F}(H(J, \g)) \cong S(\g_{\not=0}) \rtimes J$
the Hopf smashed product of $J$ with $S(\g_{\not=0})$. 

\end{proposition}

\begin{proof} This follows from the previous lemma and the PBW theorem for $U(\g)$.

\end{proof}

\subsection{PBW property for cocommutative ind-Hopf $C$ algebras in $\Ver_{p}$ with $\Delta^{-1}(C_{0} \otimes C_{0}) = \mathbf{1}$.}

In this subsection, fix $C$ to be a cocommutative Hopf algebra in $\Ver_{p}^{\ind}$ with $J = \Delta^{-1}(C_{0} \otimes C_{0}) = \mathbf{1}$. Then, $\g = \Prim(C)$ has $\g_{0} = 0$.  Note that in particular, this implies that $C$ is irreducible.

The goal of this subsection is to prove the following theorem.

\begin{theorem} \label{purelyoddcocomm} The natural map $U(\g) \rightarrow C$ is an isomorphism of Hopf algebras.

\end{theorem}

To prove this, we first need some cohomological facts.

\begin{definition} Let $C$ be a irreducible cocommutative coalgebra in $\Ver_{p}^{\ind}$ with comultiplication $\Delta$ and $\iota$ the inclusion of the unique grouplike element. The \emph{coHochschild complex} of $C$ is $\mathbf{1} \rightarrow C \rightarrow C^{\otimes 2} \rightarrow \cdots$
with the maps defined as $\iota: \mathbf{1} \rightarrow C$ and 

$$\iota \otimes \id_{C}^{\otimes n} - \Delta \otimes \id_{C}^{\otimes n-1} + \id_{C} \otimes \Delta \otimes \id_{C}^{ \otimes n-2} + \cdots +  (-1)^{n+1} \id_{C}^{\otimes n} \otimes \iota.$$

\end{definition}

If $X \in \Ver_{p}$, then $S(X)$ is a cocommuative coalgebra in $\Ver_{p}^{\ind}$ and the associated coHochschild complex is a graded complex with $S(X)$ given the natural grading. A result of Etingof in \cite{Et1} implies the following lemma.

\begin{lemma} \label{Koszulcomplex} The cohomology of the coHochschild complex in graded degree $i < p$ is $\wedge^{i}(X)$ and is concentrated in homological degree $i$.

\end{lemma}

\begin{remark} What this lemma is really saying is that Koszul duality holds in degree smaller than the characteristic, since the cohomology of the dual complex is $\Ext_{S(X)^{*}}(\mathbf{1}, \mathbf{1})$, where $S(X)^{*}$ is the graded dual to $S(X)$.

\end{remark}

As a consequence of this Lemma, we have the following result.

\begin{lemma} \label{oddKoszul} For $\g = \Prim(C)$ as defined in the section, the cohomology of the coHochschild complex of $S(\g)$ is the exterior algebra $\bigwedge(\g)$, with $\wedge^{i}(\g)$ sitting in homological degree $i$.

\end{lemma}

This follows from using the Kunneth isomorphism to reduce to a computation on each simple summand of $\g$ and then using Lemma \ref{nilpotence} to prove that $S(X)$ for each such simple summand is concentrated entirely in degrees $< p$. We can now prove the theorem.

\begin{proof}[Proof of Theorem \ref{purelyoddcocomm}] By taking associated graded under the coradical filtration (which is the PBW filtration on $U(\g)$), we can assume $C$ and $\g$ are both commutative. Consider the map $S(\g) \rightarrow C.$
This map is injective on primitives and is hence injective. So, we just need to prove it is surjective. Using injectivity, we identify $S(\g)$ with its image in $C$. We inductively show that the image contains $C(n)$, the $n$th piece of the coradical filtration on $C$. 

$$(\Delta - \id_{C} \otimes \iota - \iota \otimes \id_{C})(C(n)) \subseteq C(n-1) \otimes C(n-1)$$
and by cocommutativity, is a subset of $S^{2}(C(n-1))$, which is equal to the symmetric invariants, as we assume the characteristic is bigger than $2$ in this chapter. Hence, by induction, 

$$(\Delta - \id_{C} \otimes \iota - \iota \otimes \id_{C})(C(n)) \subseteq S^{2}(S(\g)(n-1)).$$
The image is a cocycle for the Koszul complex on $S(\g)$, but by the cohomological Lemma \ref{oddKoszul}, every symmetric cocycle is also a coboundary. Hence, for each simple object $X \in C(n)$, we can find a simple object $X \in S(\g)(n)$ such that the antidiagonal in $X \oplus X \subseteq C(n) \oplus S(\g)(n)$
is primitive, i.e., is killed by $\Delta - \id_{C} \otimes \iota - \iota \otimes \id_{C}.$
Hence, $C(n) \subseteq S(\g)(n) + C(1) = S(\g)(n) + \g.$
\end{proof}

As a consequence of the proof, we can actually state a slightly more general result.

\begin{corollary} \label{symcoalg} Let $C$ be an irreducible cocommutative coalgebra in $\Ver_{p}^{\ind}$ and let $X$ be an object in $\Ver_{p}$ that has no summand isomorphic to $\mathbf{1}$. If $\phi$ is a coalgebra map $S(X) \rightarrow C$, then $\phi$ is surjective if and only if it is surjective on primitives.

\end{corollary}

\begin{proof} The proof is identical to the theorem above, as we only use the coalgebra structure and the inclusion of the unique grouplike element.
\end{proof}

\subsection{PBW property for the coradical filtration on cocommutative Hopf algebras}

In this subsection, let $C$ be some fixed cocommutative Hopf algebra in $\Ver_{p}^{\ind}$. Let $(J, \g) = (\Delta^{-1}(C_{0} \otimes C_{0}), \Prim(C))$ be the corresponding dual Harish-Chandra pair.

From the definition of $H(J, \g)$, it is clear that there is a natural homomorphism of Hopf algebras in $\Ver_{p}^{\ind}$ $\phi: H(J, \g) \rightarrow C$
induced by the inclusion of $J$ and $\g$. The goal of this subsection is to prove the following theorem (which is a generalization of \cite[Theorem 3.6]{M2}).

\begin{theorem} \label{PBWcocomm} $\phi$ is an isomorphism.

\end{theorem}

\begin{proof} We may assume $C$ and $J$ are irreducible as coalgebras. We begin by reducing to the associated graded under relative coradical filtrations on $H(J, \g)$ and $C$ (as in Proposition \ref{relativecoradical}). For $H(J, \g)$ this filtration is the same as the PBW filtration obtained by setting $J$ in degree $0$ and $\g$ in degree $1$. Hence, by the PBW decomposition on $H(J, \g)$, we see that  $\gr(H(J, \g)) \cong S(\g_{\not=0}) \otimes \gr J$
as a Hopf algebra in $\Ver_{p}^{\ind}$. For $C$, Proposition \ref{relativecoradical} tells us that $\Prim(\gr_{F}(C)) = \g$ as a subobject of $\gr_{F}(C)$. Hence, by taking the associated graded Hopf algebra under this filtration, we reduce to the case where $C$ is an $\mathbb{N}$-graded cocommutative Hopf algebra with $C[0] = J$ and the Lie bracket on $\g_{\not=0}$ being trivial. Additionally, in this case, the homomorphism $\phi$ becomes a homomorphism $S(\g_{\not=0}) \rtimes J \rightarrow C.$

Now, $\phi$ is injective as it is injective on primitives. Hence, we just need to prove that $\phi$ is surjective. We may consider $C$ as a right $J$-comodule via the projection $\pi: C \rightarrow C[0] = J$. Let $S$ be the invariants of this coaction i.e. $S$ is the kernel of $(\id \otimes \pi) \circ \Delta - \iota_{J} \otimes \id: C \rightarrow C \otimes J.$
Then, as in \cite[Proposition 3.5]{M2}, $S$ is an irreducible cocommutative coalgebra in $\Ver_{p}^{\ind}$ and $\Prim(S) = \g_{\not=0}.$ Additionally, $\phi$ induces a map of coalgebras $S(\g_{\not=0}) \rightarrow S$ that is an isomorphism on primitives. Hence, by Corollary \ref{symcoalg}, $\phi$ induces a surjection $S(\g_{\not=0}) \rightarrow S$. 

$H(J, \g)$ is injective as a $J$-comodule, $S(\g_{\not=0})$ is the cosocle of $J$ in $H(J, \g)$ and $S$ is the cosocle of $J$ in $C$ (by the assumption of irreducibility). Hence, $\phi$ must be a surjection.
\end{proof} 

\begin{remark} Let us elaborate on the intuition behind the proof of this proposition when $C = (A^{\circ})^{1}$ for some finitely generated commutative Hopf algebra $A$. Here, $C$ is the distribution algebra on $\g = \Lie(\Spec(A))$. The proposition says that  $C = (\overline{A}^{\circ})^{1} \otimes S(\g_{\not=0})$
as a module over $(\overline{A}^{\circ})^{1}$. Normally, distribution algebras aren't enveloping algebras but rather divided power enveloping algebras. What this proposition is saying is that there are no divided powers in the part coming from $\g_{\not=0}$. This is because Lemma \ref{nilpotence} shows us that the Frobenius maps on simple objects $L_{i}$ for $i > 1$ are $0$. All of this informal divided power discussion is formally encoded in the computation of cohomology of the coHochschild complex for $S(\g_{\not=0})$.

\end{remark}

\subsection{Proof of equivalence between the categories of cocommutative Hopf algebras in $\Ver_{p}^{\ind}$ and dual Harish-Chandra pairs in $\Ver_{p}$}

\begin{theorem} \label{dHCPtheorem}

\begin{enumerate} \item[1.] Let $C$ be a cocommutative algebra in $\Ver_{p}^{\ind}$ and $(J, \g)$ the corresponding Harish-Chandra pair. Then, $H(J, \g) \cong C$. 

\item[2.] Let $(J, \g)$ be a Harish-Chandra pair in $\Ver_{p}$. Then, $\mathbf{DHC}(H(J, \g)) = (J, \g)$. 

\end{enumerate}

\end{theorem}

\begin{proof} Part 1 is Theorem \ref{PBWcocomm}. Part 2 follows from Proposition \ref{dHCPBW}.

\end{proof}

\section{\Large{\textbf{Inverse Functor for Harish-Chandra pairs: construction via duality}}}

In this section, we will give the construction of an inverse functor for Harish-Chandra pairs by first giving the definition of the inverse and then exploring some dualities that come out of the definition. 

\begin{definition} Let $(H, W)$ be a Harish-Chandra pair. Define $\A(H, W) := H \ltimes T_{c}(W_{\not=0}),$
the smash product Hopf algebra. 

\end{definition}

Note that $\A(H,W)$ is an $\mathbb{N}$-graded commutative Hopf algebra in $\Ver_{p}^{\ind}$, with the grading induced from the grading on $T_{c}(W)$. Hence, we can also define a completed version of this algebra that lives in $\Ver_{p}^{\pro}$.

\begin{definition} Define $\widehat{\A}(H, W) := \prod_{i=0}^{\infty} H \otimes T_{c}^{n}(W_{\not=0}).$

\end{definition}

\begin{proposition} Let $(H, W)$ be a Harish-Chandra pair and let $(H^{\circ}, W^{*})$ be the corresponding dual Harish-Chandra pair. Then, there is a unique non-degenerate $\mathbb{N}$-graded Hopf pairing $\A(H, W) \otimes \mathcal{H}(H^{\circ}, W^{*})$
induced from the pairings between $H$ and $H^{\circ}$ and between $T_{c}(W)$ and $T(W^{*})$.

\end{proposition}

\begin{proof} The fact that an $\mathbb{N}$-graded Hopf pairing exists and is unique is obvious. The fact that it is non-degenerate follows from the fact that each pairing is non-degenerate, which follows from Proposition \ref{non-degencirc} for $H$ and $H^{\circ}$ and the definition of tensor algebras and co-algebras.
\end{proof}

Now, $\widehat{\A}(H, W)$ is not a Hopf algebra in $\Ver_{p}^{\ind}$. The muitiplication unit, counit, and antipode maps extend without a problem but the comultiplication requires a completed tensor product, which is the natural monoidal structure on the pro-completion of $\Ver_{p}$. Hence, it is a topological pro-Hopf algebra in $\Ver_{p}$. With this structure we can make sense of the following duality statement.

\begin{proposition} The non-degenerate pairing between $\A(H, W)$ and $\H(H^{\circ}, W^{*})$, extends to a non-degenerate pairing $\widehat{\A}(H, W) \otimes \H(H^{\circ}, W^{*}) \rightarrow \mathbf{1}.$

\end{proposition}

To understand this pairing fully, we need to use the terminology of internal Homs in module categories (see \cite[Section 7.9]{EGNO}). Consider the category $\mathcal{C}^{\circ}$ of left $H^{\circ}$-modules in $\Ver_{p}^{\ind}$ that actually live inside $\Ver_{p}$. This is a module category over $\Ver_{p}$ and we have an internal Hom functor $\iHom: \mathcal{C}^{\circ} \times \mathcal{C}^{\circ} \rightarrow \Ver_{p}$
defined by the property that for any object $X \in \Ver_{p}$, $M_{1}, M_{2} \in \mathcal{C}^{\circ}$, $\mathrm{Hom}_{\mathcal{C}^{\circ}}(M_{1} \otimes X, M_{2}) = \mathrm{Hom}_{\Ver_{p}}(X, \iHom(M_{1}, M_{2})).$

This functor makes sense if $M_{1}$ is a left $H^{\circ}$-module in $\Ver_{p}^{\ind}$ as well. In this case, the internal Hom gives us a pro-object. This is because the internal Hom sends an inductive system in $M_{1}$ to the dualized projective system, which can be seen from the universal property defining the functor. Additionally, because $\Ver_{p}$ is semisimple and ind-objects are merely infinite direct sums, it also works if $M_{2}$ is an ind-object rather than an object of finite length in $\Ver_{p}$.

The reason to bring up this piece of machinery is the following fact:

\begin{proposition} Let $X$ be any object in $\Ver_{p}$ and $Y$ any right $H^{\circ}$-module in $\Ver_{p}^{\ind}$. Then, there is an isomorphism $\iHom(X \otimes H^{\circ} , Y) \cong X^{*} \otimes Y$
as objects in $\Ver_{p}^{\ind}$, with $X \otimes H^{\circ}$ given the free module structure.

\end{proposition}

\begin{proof} This follows from the defining property and the fact that for any object $Z \in \Ver_{p}$,

$$\mathrm{Hom}_{\mathcal{C}^{\circ}}(H^{\circ} \otimes X \otimes Z, Y) = \mathrm{Hom}_{\Ver_{p}}(X \otimes Z, Y) = \mathrm{Hom}_{\Ver_{p}}(Z, X^{*} \otimes Y).$$

\end{proof}

Now, if $(H, W)$ is a Harish-Chandra pair in $\Ver_{p}$, then $H$ is naturally a right $H^{\circ}$-module. The action map $a: H \otimes H^{\circ} \rightarrow H$
is defined by comultiplying in $H$ and then pairing $H^{\circ}$ with the right tensor factor. The fact that this is unital and associative as an action can be checked via the non-degenerate pairing $b$ between $H$ and $H^{\circ}$. Using the fact that this is a Hopf pairing, we can see that $b \circ (\id \otimes a) = b \circ (m \otimes \id)$
from which the properties can be deduced.

Combining the above facts, we get the following result.

\begin{proposition} \label{internalhom} Let $(H, W)$ be a Harish-Chandra pair in $\Ver_{p}$. There is an isomorphism in $\Ver_{p}^{\ind}$,

$$\xi: H \otimes T^{n}(W_{\not=0}) \rightarrow \Hom_{\mathcal{C}^{\circ}}(T^{n}(W_{\not=0}^{*}) \otimes H^{\circ}, H).$$
These isomorphisms glue together to give a pro-object isomorphism $\xi: \widehat{\A}(H, W) \rightarrow \iHom(\H(H^{\circ}, W^{*}), H).$

\end{proposition}

\begin{proposition} \label{hompairing} 

\begin{enumerate}

\item[1.] There is a natural morphism in $\Ver_{p}^{\ind}$, $\eta :  \iHom(\H(H^{\circ}, W^{*}), H) \otimes \H(H^{\circ}, W^{*})  \rightarrow H$
that is the pullback of the identity along the identification 

\begin{align*}
\iHom((\H(H^{\circ}, W^{*}), H), H) & \otimes \H(H^{\circ}, W^{*})\cong \\
&\Hom_{\Ver_{p}}(\iHom(\H(H^{\circ}, W^{*}), H), \iHom(\H(H^{\circ}, W^{*}), H)).
\end{align*}

\vspace{0.2cm}

\item[2.] Via $\xi$, the non-degenerate pairing $b$ between $\widehat{\A}(H, W)$ and $\H(H^{\circ}, W^{*})$ is identified as the following composite $ \widehat{\A}(H, W) \otimes  \H(H^{\circ}, W^{*}) \rightarrow H \rightarrow H \otimes H^{\circ} \rightarrow \mathbf{1},$ where the first map is $\eta$, the second map is the inclusion of the unit into $H^{\circ}$ and the third map is the pairing between $H^{\circ}$ and $H$.

\vspace{0.2cm}

\item[3.] If $M$ is a right $H^{\circ}$-submodule of $\H(H^{\circ}, W^{*})$, then $\xi$ identifies $\iHom_{\mathcal{C}^{\circ}}(\H(H^{\circ}, W^{*})/M, H)$
with $M^{\perp}$ under the pairing with $\widehat{\A}(H, W)$.

\end{enumerate}

\end{proposition}

\begin{proof}

\begin{enumerate}

\item[1.] This is just the universal property of $\iHom$.

\item[2.]  Note that the piece of $\eta$ in graded degree $n$ $\eta_{n}:  \Hom_{\mathcal{C}^{\circ}}( T^{n}(W_{\not=0})^{*} \otimes H^{\circ}, H) \otimes T^{n}(W_{\not=0}^{*}) \otimes H^{\circ} \rightarrow H$
is obtained by identifying the left tensor factor with $H \otimes T^{n}(W_{\not=0})$ (as in the previous proposition) and then pairing $T^{n}(W^{*}_{\not=0})$ with $T^{n}(W_{\not=0})$ while acting by $H^{\circ}$ on $H$ via the right module structure. To see this, look at the following argument. The left tensor factor of the domain of $\eta_{n}$ in $\Ver_{p}$ is just $H \otimes T^{n}(W_{\not=0})$. The identity map on this space gets identified with the map $\eta'_{n}: T^{n}(W_{\not=0}) \otimes T^{n}(W)^{*}_{\not=0} \otimes H \rightarrow H$
given by evaluation on the first two tensor factors followed by identity on the third tensor factor. This is the same as the map $\eta_{n}$ restricted to the unit in $H^{\circ}$ (in the right tensor factor). But since $\eta_{n}$ is a $H^{\circ}$-module map, it suffices to compute it on the unit.

The proof of part 2 is now easy. Inside $\eta$, we have already done the pairing between the tensor algebras. The pairing between $H^{\circ}$ and $H$ remains and this comes from the fact that the evaluation pairing between $H$ and $H^{\circ}$ is the same as the map obtained by acting on $H$ by $H^{\circ}$ (which happens inside $\eta$) and then evaluating with the unit in $H^{\circ}$.

\item[3.] We can identify $\iHom(\H(H^{\circ}, W^{*})/M, H)$
as the complement of $M$ under the $\eta$ pairing. By part 2, it is clear that this sits inside $M^{\perp, b}$, the complement of $M$ under the non-degenerate pairing $b$. Let $K = M^{\perp, b}$. Then, the image of $K \otimes M$ under $\eta$ is a $H^{\circ}$-submodule of $H$ that is a subobject of the complement of the image of $\iota_{H^{\circ}}$ under the evaluation pairing between $H$ and $H^{\circ}$. However, this implies that

$$0 = \ev(\eta (K \otimes M) \otimes \mathrm{im}(\iota_{H^\circ})) = \ev(\eta (K \otimes M) \cdot H^{\circ} \otimes \mathrm{im}(\iota_{H^{\circ}})) = \ev(\eta(K \otimes M) \otimes H^{\circ})$$
and hence $\eta(K \otimes M) = 0$ by non-degeneracy of $\ev.$  This proves the reverse inclusion.

\end{enumerate}

\end{proof}

Using these dualities, we can finally construct a potential quasi-inverse to $\mathbf{HC}$.

\begin{definition} Let $(H, W)$ be a Harish-Chandra pair in $\Ver_{p}^{\ind}$. Recall the construction of $H(H^{\circ}, W^{*})$ as the quotient of $\H(H^{\circ}, W^{*})$ by an ideal $I(H^{\circ}, W^{*})$. Define $A(H, W) \subseteq \widehat{\A}(H, W)$
as the complement of $I(H^{\circ}, W^{*})$ under the non-degenerate pairing between $\H(H^{\circ}, W^{*})$ and $\widehat{\A}(H, W).$

\end{definition}

Here are some properties of $A(H, W)$.

\begin{lemma} Keeping the notation from the above definition,

\begin{enumerate}

\item[1.] $A(H, W)$ is a subalgebra in the topological algebra $\widehat{\A}(H, W)$ and is stable under the antipode.

\item[2.] $A(H, W)$ is discrete in $\widehat{\A}(H, W)$. Moreover the coproduct on $\widehat{\A}(H, W)$ induces a coproduct 

$A(H, W) \rightarrow A(H, W) \otimes A(H, W).$

\item[3.] $A(H, W)$ is a commutative ind-Hopf algebra in $\Ver_{p}.$

\end{enumerate}

Hence, $A$ defines a functor from the category of Harish-Chandra pairs in $\Ver_{p}$ to the category of commutative ind-algebras in $\Ver_{p}$.

\end{lemma}

\begin{proof} The proof of this lemma is identical to that of \cite[Lemma 4.20]{M1}, which is the corresponding lemma for supervector spaces.

\end{proof}

\begin{proposition} \label{HCPBW} Let $(H, W)$ be a Harish-Chandra pair, let $(J, \g) = (H^{\circ}, W^{*})$ be the corresponding dual Harish-Chandra pair. Let $i_{\g}: S(\g_{\not=0}) \rightarrow T(\g_{\not=0})$ be the inclusion of $S(\g_{\not=0})$ as the subalgebra of invariants under the braiding $c$ (which makes sense by Lemma \ref{nilpotence}). Let $\phi_{\g}$ be the unit-preserving isomorphism of left $J$-module coalgebras $ S(\g_{\not=0}) \rtimes J \rightarrow H(J, \g)$
induced by $i_{\g}$. Define the map 

$$\psi_{W} : A(H, W) \rightarrow \widehat{\A}(H, W) = H \widehat{\otimes} T_{c}(W_{\not=0}) \rightarrow H \otimes S(W_{\not=0})$$
where the first map is the inclusion and the last map is the natural projection $\id_{H} \otimes \pi_{W}.$ Let $b$ be the non-degenerate Hopf pairing between $A(H, W)$ and $H(J, \g)$. 

\begin{enumerate}

\item[1.] $\psi_{W}$ is a counit preserving isomorphism of ind-algebras in $\Ver_{p}$, such that

$$b \circ (\phi_{\g} \otimes \id_{A(H, W)}) = b \circ (\id_{H(J, \g)} \otimes \psi_{W}).$$

\item[2.] $A(H, W)$ is a finitely generated commutative ind-Hopf algebra in $\Ver_{p}$.

\end{enumerate}

\end{proposition}

\begin{proof} The proof of this theorem is identical to the proof \cite[Lemma 4.21]{M1}. Proposition \ref{hompairing} allows us to prove that the isomorphism $\xi: \widehat{\A}(H, W) \rightarrow \iHom_{J}(\H(J, \g), H)$
of Proposition \ref{internalhom} restricts to an isomorphism $A(H, W) \rightarrow \iHom_{J}(H(J, \g), H).$ Additionally, by the PBW property of dual Harish-Chandra pairs, $\phi_{\g}$ is an isomorphism. The rest of the proof follows identically to the one in \cite{M1}. Here we use the fact that $S(\g^{*}_{\not=0}) = S(\g_{\not=0})^{*}$ via Lemma \ref{nilpotence}.

\end{proof} 

\begin{theorem} \label{HCPPBW} Let $(H, W)$ be a Harish-Chandra pair in $\Ver_{p}$. Then, the Harish-Chandra pair corresponding to $A(H, W)$ is naturally isomorphic to $(H, W)$, i.e., the constructed functor going from the category of Harish-Chandra pairs in $\Ver_{p}$ to the category of affine group schemes of finite type in $\Ver_{p}$ is right inverse to the functor going in the other direction.

\end{theorem}

\begin{proof} Recall the definition of the underlying ordinary commutative algebra: given a commutative algebra $A$ in $\Ver_{p}$, this is the commutative $k$-algebra $A/\langle A_{\not=0}\rangle$. It is clear from Proposition \ref{HCPBW} that the underlying ordinary commutative algebra associated to $A(H, W)$ is $H$. We need to check that $W$ is the dual to the Lie algebra of $A(H, W)$. Let $(J, \g) = (H^{\circ}, W^{*})$ be the corresponding dual Harish-Chandra pair. Using Theorem \ref{dHCPtheorem}, we just need to check that $H(J, \g) = A(H, W)^{\circ}.$ This follows in the exact same manner as \cite[Proposition 4.22]{M1}.

\end{proof}

In order to prove that the two functors are fully inverse to each other, we need to study the geometry of $G = \Spec(A)$ a little more, for $G$ an affine group scheme of finite type in $\Ver_{p}$.

\begin{lemma} \label{formalsmooth} Let $G$ be an affine group scheme of finite type in $\Ver_{p}$ corresponding to the finitely generated commutative Hopf algebra $\O(G)$. Let $G_{0}$ be the underlying ordinary affine group scheme, corresponding the the Hopf algebra $\O(G_{0}) = \O(G)/\langle\O(G)_{\not=0}\rangle$. Let 

$$\widehat{\rho}: \widehat{\O(G)}_{\id} \rightarrow S(\g^{*}_{\not=0})$$
be any algebra homomorphism dual to an inclusion of $S(\g_{\not=0}) \rightarrow \O(G)^{\circ}$ that determines a PBW isomorphism $\O(G)^{\circ} \cong \O(G_{0})^{\circ} \otimes S(\g_{\not=0})$. Define a homomorphism

$$\widehat{\O(G)}_{\id} \cong    \widehat{\O(G_{0})}_{\id} \otimes S(\g^{*}_{\not=0}) $$
by comultiplying in $\widehat{\O(G)}_{\id}$, using the completed tensor product, followed by applying $\widehat{\rho}$ in the right tensor and the natural projection in the left tensor. Then, this homomorphism is an isomorphism of left $\widehat{\O(G_{0})}_{\id}$-comodule algebras in $\Ver_{p}^{\pro}$.

\end{lemma}

\begin{proof} This follows from the PBW decomposition on $(\O(G)^{\circ})^{1}$ and the fact that $((\O(G)^{\circ})^{1})^{*} \cong \widehat{\O(G)}_{\id}.$

\end{proof} 

Our goal is to show that such an isomorphism exists before completing the algebras as well. To do so, we will need the following result that is a direct generalization of Radford's Bosonization to the setting of $\Ver_{p}$. 

\begin{proposition} \label{boson} Let $A$ and $H$ be ind-Hopf algebras in $\Ver_{p}$ and suppose we have Hopf algebra maps $\pi: A \rightarrow H$ and $i: H \rightarrow A$ such that $\pi \circ i = \id_{H}$. Define $\rho: A \rightarrow A$ as $m_{A} \circ (\id_{A} \otimes (i \circ S_{H} \circ \pi)) \circ \Delta_{A}$ and define $B = \rho(A)$. Then:

\begin{enumerate}

\item[(a)] $B$ is a left $H$-module under the adjoint action of $i(H)$ and is a left $H$-comodule under the coadjoint action via $\pi$.

\item[(b)] $B$ is a left $H$-comodule Hopf algebra and $\rho$ is a Hopf algebra map. 

\item[(c)] The map $m \circ (i_{H} \otimes i_{B}) : H \ltimes B \rightarrow A$ is an isomorphism of left $H$-comodule Hopf algebras.

\end{enumerate}

\end{proposition} 

\begin{proof} This lemma is identical to Theorem 3 in \cite{R}. That theorem, along with the prerequisite Theorem 1 in the same paper, have proofs that can be written entirely in terms of the structural Hopf algebra and module-comodule morphisms of $H, A$ and $B$ and can thus be generalized to the categorical setting with no change.

\end{proof}

We will also need the following result regarding the associated graded of $A$ under a suitable filtration.

\begin{lemma} \label{descfilt} Let $A$ be a commutative ind-Hopf algebra in $\Ver_{p}$. Let $J$ be the ideal in $A$ generated by $A_{\not=0}$ and let $F$ be the descending $J$-adic filtration on $A$. Then,$\gr_{F}A$ is a commutative ind-Hopf algebra in $\Ver_{p}$ and 

$$(\mathrm{gr}_{F}A)^{\circ} \cong \mathrm{gr}_{F}(A^{\circ})$$
where the filtration $F$ on $A^{\circ}$ is the relative coradical filtration from Proposition \ref{relativecoradical}.

\end{lemma}

\begin{proof} For convenience of notation, let $C$ denote $A^{\circ}$ and $D$ denote $\Delta^{-1}(C_{0} \otimes C_{0})$. Note that 

$$\mathrm{gr}_{F}A = \bigoplus_{i=0}^{\infty} J^{i}/J^{i+1}.$$
Using the perfect pairing between $C$ and $A$, we can identify $(J^{i}/J^{i+1})^{\circ}$ with $K_{i+1}/K_{i}$, where $K_{i} \subseteq C$ is the complement of $J^{i}$ under the pairing. So, we just need to show that $F_{i}(C) = K_{i+1}$ for all $i$. 

We show this via induction on $i$. Let $b$ be the pairing map. The base case is fairly straightforward. Comparing isotypic components, the complement of $A_{\not=0}$ must be $C_{0}$. $J$ is the image under multiplication of $A \otimes A_{\not=0}$. Hence, as $b$ is a Hopf pairing, $K_{1}$, the complement under $b$ of $J$, is the subobject of $C$ that satisfies $\Delta(K_{1}) \subseteq C \otimes C_{0}$, which by cocommutativity is the same as $\Delta^{-1}(C_{0} \otimes C_{0})$. Hence, $K_{1} = F_{0}(C) = D.$

Assume now for sake of induction that $K_{i+1} = F_{i}(C)$ for $i < n$. Note that $J^{n}$ is the image under multiplication in $A$ of $J^{n-1} \otimes J$. Under the tensor product pairing, the complement of $J \otimes J^{n-1}$ is 

$$K_{n-1} \otimes C \oplus C \otimes K_{1} = F_{n-2}(C) \otimes C \oplus C \otimes F_{0}(C)$$
by the inductive hypothesis. Hence, $K_{n} = \Delta^{-1}(F_{n-2}(C) \otimes C \oplus C \otimes F_{0}(C)) = F_{n-1}(C).$ by the inductive definition of the relative coradical filtration.

\end{proof}

\begin{lemma} \label{smoothness} If $A$ is a finitely generated commutative ind-Hopf algebra in $\Ver_{p}$, then $A \cong  \overline{A} \otimes S(\g^{*}_{\not = 0})$
as a left $\overline{A}$-comodule algebra.

\end{lemma}

\begin{proof} Recall that $\overline{A}$ is $A/\langle A_{\not=0} \rangle$. Pick some isomorphism $\widehat{A}_{\id} \cong \widehat{\overline{A}}_{\id} \otimes S(\g^{*}_{\not=0})$ as in Lemma \ref{formalsmooth}. We begin by defining an algebra homomorphism $\rho: A \rightarrow S(\g^{*}_{\not=0})$ by composing the natural map from $A \rightarrow \widehat{A}_{\id}$ with the projection from $\widehat{A}_{\id} \rightarrow S(\g^{*}_{\not=0})$ obtained by taking the quotient by the ideal generated by the maximal ideal of $\widehat{\overline{A}}_{\id}.$ We use this to define a homomorphism of algebras 

$$\eta: A \rightarrow  \overline{A}\otimes S(\g^{*}_{\not=0}) $$
by first comultiplying $A \rightarrow A \otimes A$ and then applying $\rho$ in the left tensor component and the natural projection $A \rightarrow \overline{A}$ in the right tensor component. $\eta$ is a homomorphism of left $\overline{A}$-comodule algebras. We will show $\eta$ is an isomorphism. 

To do this, we will show that the associated graded of $\eta$ under the suitable descending filtration on $A$ and $\overline{A} \otimes S(\g^{*}_{\not=0})$ is an isomorphism. Let $J$ be the ideal in $A$ generated by $A_{\not=0}$. Let $F$ be the $J$-adic filtration on $A$, as in the previous lemma, and let $F'$ be the descending filtration on $S(\g^{*}_{\not=0}) \otimes \overline{A}$ determined by the ideal generated by $\g^{*}_{\not=0}$. It is clear that $\eta$ is filtration preserving, since $A_{\not=0}$ must land inside the ideal generated by $\g^{*}_{\not=0}$. by semi-simplicity of $\Ver_{p}$. Additionally, by Lemma \ref{descfilt} and Proposition \ref{relativecoradical}, $\g$ is the Lie algebra of $\mathrm{gr}_{F}(A)$. Hence, by taking associated graded, we reduce to the case where $A$ is a graded Hopf algebra with $\overline{A}$ as the degree $0$ term.

In this setting, we have Hopf algebra maps $\pi: A \rightarrow \overline{A}$ and $i: \overline{A} \rightarrow A$ such that $\pi \circ i = \id_{\overline{A}}$. Hence, we can apply Lemma \ref{boson} to find a left $\overline{A}$-comodule Hopf algebra $B$ in $\Ver_{p}^{\ind}$ such that $A \cong \overline{A} \ltimes B$. Additionally, we know from Lemma \ref{nilpotence} that $A$ has finite length as a left $\overline{A}$-module. Hence, $B$ is actually an object in $\Ver_{p}$ and not an ind-object. 

Now, by the PBW decomposition on $(A^{\circ})^{1}$, we have

$$(\overline{A}^{\circ})^{1} \ltimes S(\g_{\not=0})  \cong (A^{\circ})^{1} \cong(\overline{A}^{\circ})^{1} \ltimes B^{*}$$
as left $(\overline{A}^{\circ})^{1}$-module Hopf algebra.

Hence, $B \cong S(\g_{\not=0})^{*} \cong S(\g^{*}_{\not=0})$ as a left $\overline{A}$-comodule Hopf algebra, with comultiplication in $S(\g^{*}_{\not=0})$ determined by setting $\g^{*}_{\not=0}$ primitive.

\end{proof}

We can now finish the proof that the functor $\mathbf{A}$ that sends $(H, W)$ to $A(H, W)$ is a quasi-inverse to $\mathbf{HC}$.

\begin{theorem} \label{HCPtheorem} Let $A$ be a  finitely generated commutative ind-Hopf algebra in $\Ver_{p}$. Then, $A(\overline{A}, \g^{*})$ is naturally isomorphic to $A$.

\end{theorem}

\begin{proof} This follows in essentially the same manner as \cite[Theorem 4.23]{M1}, with all the prerequisites for the proof being taken care of in previous sections. We reprove it here for convenience of the reader. 

Let $(\overline{A}, \g^{*}) = \mathbf{HC}(A).$ Let $H = \overline{A}, W = \g^{*}_{\not=0}$. Let $I$ be the augmentation ideal of $A$. Note that $W = (I/I^{2})_{\not=0}$. Let $\omega$ be the canonical projection of $A$ onto $W$, i.e., the composite of the projection onto $I$ followed by the projection onto $I/I^{2}$ and then the projection onto the part not coming from vector spaces. 

Define $\omega^{(n)}: A \rightarrow T^{n}(W)$ as $n$-fold comultiplication followed by $\omega$ in each component, for $n > 0$ and $\epsilon$ for $n = 0$. Finally, define $\beta: A \rightarrow \widehat{\A}(H, \g^{*})$ as comultiplication followed by $\sum_{n} \id \otimes \omega^{(n)}.$ It suffices to prove that $\beta$ gives a Hopf algebra isomorphism from $A$ onto $A(H, \g^{*})$. 

Let $C = A^{\circ}$, $(J, \g) = \mathbf{DHC}(C, \g^{*})$ and $V = \g_{\not=0}$. Lemma \ref{smoothness} and the fact that $S(W^{*}) = S(W)^{*}$ (from Lemma \ref{nilpotence}) immediately imply that $C = S(V) \otimes \overline{A}^{\circ}$
and hence $J = \overline{A}^{\circ}$ and $V = W^{*}$ as a $J$-module. Additionally, from Theorem \ref{dHCPtheorem}, we have an isomorphism $H(J, \g) \rightarrow C$
induced by the natural maps from $J$ into $C$ and $T(V)$ into $C$. Let $\gamma: \H(J, \g) \rightarrow C$
be the composition of this isomorphism with the natural projection of $\H(J, \g)$ onto $H(J, \g)$. This is a Hopf algebra homomorphism in $\Ver_{p}^{\ind}$. It is easy to see that $\gamma$ is adjoint to $\beta$ via the non-degenerate pairings between $\H(J, \g)$ and $\A(C, \g^{*})$ and $H$ and $A$. Hence, $\beta$ is injective and maps into $A(H, W)$ as $\gamma$ kills $I(J, \g)$. 

Moreover, Theorem \ref{HCPPBW} implies that we have an isomorphism of left $\overline{A}$-comodules $\rho: A(H, W) \cong H \otimes S(W).$
Now, consider the following commutative diagram:

$$\begin{tikzpicture}
\matrix (m) [matrix of math nodes,row sep=5em,column sep=5em,minimum width=5em]
{A & A(H, W)  \\
H \otimes S(W)& H \otimes S(W)\\};;
\path[-stealth]
(m-1-1) edge node[auto] {$\beta$} (m-1-2)
(m-1-2) edge node[auto] {$\rho$} (m-2-2)
(m-2-1) edge node[auto] {$i$} (m-1-1)
(m-2-1) edge node[auto] {$\phi$} (m-2-2);
\end{tikzpicture}$$
where $i$ is inverse to an isomorphism constructed from Lemma \ref{smoothness}. Restricting $\phi$ to $\mathbf{1} \otimes S(W)$, we see that $\phi$ is simply a section $S(W) \rightarrow T(W)$ (which makes sense as $S(W)$ has no $p$th powers), followed by $\beta$ followed by $\rho$, which is just the identity. Hence, $\phi$ is an injection of cofree $H$-comodules that contains $S(W)$ in the image and is hence an isomorphism. Thus, $\beta$ is an isomorphism $A \rightarrow A(H,W)$ as desired. 

\end{proof}

\begin{corollary} The category of affine group schemes of finite type in $\Ver_{p}$ is equivalent to the category of Harish-Chandra pairs in $\Ver_{p}$.

\end{corollary}

\begin{corollary} Let $G$ be an affine group scheme of finite type in $\Ver_{p}$. Then the set of subgroup schemes of $G$ corresponds to the set 

$\{(H_{0}, \h): H_{0} \text{ a subgroup of } G_{0}, \, \h \text{ a Lie subalgebra of } \g, \, \h \text{ closed under the adjoint action of } H\}.$

\end{corollary} 

This corollary follows immediately from the correspondence between Harish-Chandra pairs and affine group schemes of finite type in $\Ver_{p}$.

\subsection{Affine group schemes in $\Ver_{p}$ with trivial underlying ordinary group} \label{odd}

In this section, we will analyze those affine group schemes $G$ of finite type in $\Ver_{p}$ such that $G_{0} = \Spec(\k)=:1$ is the trivial group. Note that by Lemma \ref{nilpotence}, such groups have function algebras in $\Ver_{p}$ rather than in $\Ver_{p}^{\ind}$ and hence are finite group schemes with only one closed point. We begin by constructing some examples.

\begin{example} Let $\g$ be a Lie algebra in $\Ver_{p}$ with $\g_{0} = 0$. Define $\O(G) = U(\g)^{*}$, which is in $\Ver_{p}$ as $U(\g)$ has finite length by Lemma \ref{nilpotence} and the PBW theorem for $\g$. Then, $G$ is a finite group scheme in $\Ver_{p}$ with $G_{0} = 1$. 

\end{example}

In \cite{Et1}, Etingof shows that every operadic Lie algebra $\g$ in $\Ver_{p}$ with $\g_{0} = 0$ is a Lie algebra and satisfies PBW. Hence, $\g$ injects into $U(\g)$ and we have the following proposition.

\begin{proposition} $Lie(G) = \g$. Thus, $\O(G) \cong S(\g^{*})$ as an algebra with counit, and $U(\g) = \O(G)^{\circ}.$

\end{proposition}

Moreover, these are all the finite group schemes in $\Ver_{p}$ with trivial $G_{0}$.

\begin{theorem} \label{oddclassify} Let $G$ be an affine group scheme of finite type in $\Ver_{p}$ with trivial $G_{0}$. Let $\g$ be its Lie algebra. Then, $\O(G) \cong U(\g)^{*}$ as a commutative Hopf algebra in $\Ver_{p}$.

\end{theorem}

\begin{remark} This is also directly provable from Lemma \ref{smoothness}, as it is the special case where $\g_{\not=0} = 0.$

\end{remark}

This follows from the dual result for cocommutative Hopf algebras in $\Ver_{p}$ stated in Theorem \ref{purelyoddcocomm}.

Hence, we see that the correspondence between Harish-Chandra pairs and affine group schemes in $\Ver_{p}$ is just the correspondence between a Lie algebra and its enveloping algebra, if the underlying ordinary affine group scheme is trivial.

\section{\Large{\textbf{{Representations of affine group schemes of finite type in $\Ver_{p}$}}}}

We apply Theorem \ref{HCPtheorem} to the representation theory of affine group schemes of finite type in $\Ver_{p}$.

\begin{definition} Let $V$ be an object in $\Ver_{p}$. A \emph{representation} of $G$ on $V$ is a right $\O(G)$-comodule structure on $V$.

\end{definition}

\begin{remark} If we think of $G$ as a functor from the category of commutative algebras in $\Ver_{p}^{\ind}$ to the category of groups, a representation $V$ of $G$ may be thought of as a functor from the category of commutative algebras in $\Ver_{p}^{\ind}$ to the category of vector spaces and a natural transformation $G \times V \rightarrow V$ such that for any commutative algebra $A$, $V(A)$ acquires the structure of a representation of $G(A)$.

\end{remark}

\begin{remark}  If $V$ is a representation of $G$, then $V$ is also a left $\O(G)^{\circ}$-module. The converse is not necessarily true, however.

\end{remark}

\begin{definition} Let $\g$ be a Lie algebra in $\Ver_{p}$. A representation of $\g$ in $\Ver_{p}$ is an object $V$ equipped with a map $a: \g \otimes V \rightarrow V$
such that 

$$a \circ ([-,-]_{\g} \otimes \id_{V}) - (a \circ (\id_{\g} \otimes a) \circ ((\id_{\g \otimes \g}- c_{\g, \g}) \otimes \id_{V})) = 0$$
as a map from $\g \otimes \g \otimes V \rightarrow V.$

\end{definition}

Note that for any object $V$, the object $V \otimes V^{*}$ is an associative algebra with unit given by $\coev_{V}$ and multiplication given by $\ev_{V}$ in the middle in $V \otimes (V^{*} \otimes V) \otimes V^{*}$. 

\begin{definition} The Lie algebra $\gl(V)$ is $V \otimes V^{*}$ equipped with the commutator. \end{definition}

The following facts follow in exactly the same manner as in the standard case. 

\begin{proposition} A Lie algebra representation of $\g$ on $V$ is equivalent to a Lie algebra homomorphism $\g \rightarrow \gl(V)$.

\end{proposition}

\begin{proposition} If $A$ is an associative, unital ind-algebra in $\Ver_{p}$, a left $A$-module structure on $V$ is equivalent to an associative algebra homomorphism from $A$ to $\gl(V)$ in $\Ver_{p}^{\ind}$. 

\end{proposition}

\begin{proposition} If $V$ is a representation of an affine group scheme of finite type $G$ in $\Ver_{p}$, then $V$ is also a representation of $\g$ in a canonical manner. 

\end{proposition}

We can now define representations of Harish-Chandra pairs.

\begin{definition} Let $(J, \g)$ be a dual Harish-Chandra pair in $\Ver_{p}$ with the left $J$-action on $\g$ given by $\rho: J \otimes \g \rightarrow \g$. A representation of $(J, \g)$ in $\Ver_{p}$ is an object $V \in \Ver_{p}$ equipped with 

\begin{enumerate}

\item[1.] A left $J$-module structure $a: J \otimes V \rightarrow V$

\item[2.] A left $\g$-module structure $b: \g \otimes V \rightarrow V.$

\item[3.] A compatibility relation: the diagram 

$$\begin{tikzpicture}
\matrix (m) [matrix of math nodes,row sep=5em,column sep=5em,minimum width=5em]
{ J \otimes \g \otimes V & J \otimes V \\
 J \otimes J \otimes \g \otimes V & V \\
J \otimes \g \otimes J \otimes V & V\\};;
\path[-stealth]
(m-1-1) edge node[auto] {$\id_{J} \otimes b$} (m-1-2)
(m-1-2) edge node[auto] {$a$} (m-2-2)
(m-1-1) edge node[auto] {$\Delta_{J} \otimes \id_{\g\otimes V}$} (m-2-1)
(m-2-1) edge node[auto] {$\id_{J} \otimes c_{J, \g} \otimes \id_{V}$} (m-3-1)
(m-3-1) edge node[auto] {$\rho \otimes a$} (m-3-2)
(m-3-2) edge node[right] {$a$} (m-2-2);
\end{tikzpicture}$$
commutes. This is equivalent to the action map $b: \g \otimes V \rightarrow V$ being $J$-equivariant.

\item[4.] The two actions of $\g_{0}$ on $V$ induced via restriction to $\Prim(J)$ from $J$ and the restriction from $\g$ to $\g_{0}$ coincide.

A \emph{homomorphism} of dual Harish-Chandra pair representations $V \rightarrow W$ is a morphism in $\Ver_{p}$ that is both a $J$-module homomorphism and a $\g$-module homomorphism.

\end{enumerate}

\end{definition}

Note that the compatibility relation, along with the equality of the two restrictions to $\g_{0}$, is defined in the precise manner needed to define a homomorphism $\H(J, \g) \rightarrow \gl(V)$. Moreover, since $V$ is a Lie algebra representation of $\g$, the ideal $I(J, \g)$ is contained in the annihilator of $V$. Hence, we get a representation of $H(J, \g)$. An immediate consequence of Theorem \ref{dHCPtheorem} is the following.

\begin{corollary} \label{repdHCP} The category of left modules in $\Ver_{p}$ of a cocommutative ind-Hopf algebra in $\Ver_{p}$ is equivalent to the category of representations of the associated dual Harish-Chandra pair.

\end{corollary}

Of course, we really want to understand representations of $G$ and not just left modules for $\O(G)^{\circ}$. These are not necessarily the same thing, $G$-representations are integrable representations of $\O(G)^{\circ}$. This means we need to define representations for Harish-Chandra pairs and not just their duals.

\begin{definition} Let $(G_{0}, \g^{*})$ be a Harish-Chandra pair in $\Ver_{p}$. A representation of this pair is an object $V \in \Ver_{p}$ equipped with 

\begin{enumerate}

\item[1.] The structure of a $G_{0}$-representation on $V$, or equivalently, the structure of a right $\O(G_{0})$-module on $V$

\item[2.] The structure of a $\g$-module on $V$

\end{enumerate}

such that the $\O(G_{0})^{\circ}$ and $\g$-module structures on $V$ satisfy the compatibility relation of a dual Harish-Chandra pair representation. 

\end{definition} 

Note that the twisted coalgebra structure on $\A(\O(G_{0}), \g^{*})$ is defined in a dual manner to the twisted algebra structure on $\H(\O(G_{0})^{\circ}, \g)$ and hence the compatibility relation for dual Harish-Chandra pair representations implies the following proposition.

\begin{proposition} If $V$ is a representation of $(G_{0}, \g^{*})$, then there is a unique right $\widehat{\A}(\O(G_{0})^{\circ}, \g^{*})$-comodule structure on $V$ whose projection to $T_{c}(\g^{*})$ and $\O(G_{0})$ respectively induce the structures of a left $\g$-module on $V$ and a right $\O(G_{0})$-comodule on $V$ involved in the definition of a representation of the Harish-Chandra pair on $V$.

\end{proposition}

\begin{corollary} The category of representations of an affine group scheme of finite type in $\Ver_{p}$ is equivalent to the category of representations of the associated Harish-Chandra pair in $\Ver_{p}$.

\end{corollary}

\begin{proof} Via Corollary \ref{repdHCP} and Theorem \ref{HCPtheorem}, it suffices to show that the this coaction of $\widehat{\mathcal{A}}(\O(G_{0}), \g^{*})$ factors through $A(\O(G_{0}), \g^{*}).$ To see this, note that the following diagram commutes:

$$\begin{tikzpicture}
\matrix (m) [matrix of math nodes,row sep=5em,column sep=5em,minimum width=5em]
{ \H(\O(G_{0})^{\circ}, \g) \otimes V & \H(\O(G_{0})^{\circ}, \g) \otimes \widehat{\mathcal{A}}(\O(G_{0}), \g^{*}) \otimes V  \\
& V\\};;
\path[-stealth]
(m-1-1) edge node[auto] {$\id_{\H} \otimes \rho$} (m-1-2)
(m-1-2) edge node[auto] {$\langle -, -\rangle \otimes \id_{V} $} (m-2-2)
(m-1-1) edge node[auto] {$a$} (m-2-2);
\end{tikzpicture}$$
Here, $\rho: V \rightarrow V \otimes \widehat{\mathcal{A}}(\O(G_{0}), \g^{*})$ is the coaction map from the previous proposition, $a: \H(\O(G_{0})^{\circ}, \g) \otimes V \rightarrow V$ is the dual action map and $\langle -, - \rangle$ is the pairing between $\widehat{A}$ and $\H$. Since the diagonal map is $0$ when restricted to $\I(\O(G_{0})^{\circ}, \g)$ and the pairing is non-degenerate, the action map $\rho$ must factor through the orthogonal complement $A(G_{0}, \g^{*})$, as desired. 

\end{proof} 

\begin{remark} In a followup paper, we will use this theorem to classify irreducible representations of the group schemes $GL(X)$ corresponding to objects $X$ in $\Ver_{p}$.

\end{remark}

\end{document}